\documentclass[11pt,twoside]{amsart}
\usepackage{amssymb}

\usepackage[latin1]{inputenc}
\usepackage[T1]{fontenc}
\usepackage{mathtools}
\usepackage{graphicx}
\usepackage{tikz}
\usepackage{mathabx}
\usetikzlibrary{chains}
\usepackage[OT2,T1]{fontenc}

\PassOptionsToPackage{pdfusetitle,pagebackref,colorlinks}{hyperref}
\usepackage{bookmark}
\hypersetup{
  linkcolor={red!70!black},
  citecolor={green!70!black},
  urlcolor={blue!80!black}
}

\usepackage[latin1]{inputenc}
\usepackage{amsmath}
\usepackage{amsthm}

\usepackage[all]{xy}
\usepackage{hyperref}
\newtheorem{thm}{Theorem}[section]

\newtheorem{prop}[thm]{Proposition}
\newtheorem{defprop}[thm]{Definition--Proposition}

\newtheorem{lem}[thm]{Lemma}

\numberwithin{equation}{section}

\theoremstyle{definition}
\newtheorem{definition}[thm]{Definition}
\newtheorem{remark}[thm]{Remark}

\newtheorem{ex}[thm]{Example}

\DeclareFontFamily{U}{mathc}{}
\DeclareFontShape{U}{mathc}{m}{it}%
{<->s*[1.03] mathc10}{}
\DeclareMathAlphabet{\mathcal}{U}{mathc}{m}{it}
\newcommand{\kend}{\mathcal{E\mkern-3mu nd}}

\newcommand{\ch}{{\rm ch}}
\newcommand{\td}{{\rm td}}

\newcommand{\Br}{{\rm Br}}
\newcommand{\SBr}{{\rm SBr}}
\newcommand{\SBro}{{\rm SBr}^{\rm o}}
\newcommand{\Coh}{{\rm Coh}}

\newcommand{\NS}{{\rm NS}}
\newcommand{\Pic}{{\rm Pic}}

\newcommand{\rk}{{\rm rk}}

\newcommand{\Hom}{{\rm Hom}}

\newcommand{\cal}{\mathcal}
\newcommand{\ka}{{\cal A}}

\newcommand{\kc}{{\cal C}}

\newcommand{\kl}{{\cal L}}

\newcommand{\km}{{\cal M}}
\newcommand{\ko}{{\cal O}}
\newcommand{\kp}{{\cal P}}

\newcommand{\ks}{{\cal S}}

\newcommand{\GG}{\mathbb{G}}

\newcommand{\ZZ}{\mathbb{Z}}
\newcommand{\QQ}{\mathbb{Q}}

\newcommand{\CC}{\mathbb{C}}

\newcommand{\PP}{\mathbb{P}}

\DeclareSymbolFont{cyrletters}{OT2}{wncyr}{m}{n}
\DeclareMathSymbol{\Sha}{\mathalpha}{cyrletters}{"58}

\renewcommand{\to}{\xymatrix@1@=15pt{\ar[r]&}}
\newcommand{\lto}{\xymatrix@1@=15pt{&\ar[l]}}
\renewcommand{\rightarrow}{\xymatrix@1@=15pt{\ar[r]&}}
\renewcommand{\mapsto}{\xymatrix@1@=15pt{\ar@{|->}[r]&}}
\newcommand{\mapslto}{\xymatrix@1@=15pt{&\ar@{|->}[l]&}}
\renewcommand{\twoheadrightarrow}{\xymatrix@1@=18pt{\ar@{->>}[r]&}}
\renewcommand{\hookrightarrow}{\xymatrix@1@=15pt{\ar@{^(->}[r]&}}
\newcommand{\hook}{\xymatrix@1@=15pt{\ar@{^(->}[r]&}}
\newcommand{\congpf}{\xymatrix@1@=15pt{\ar[r]^-\sim&}}
\renewcommand{\cong}{\simeq}

\newcommand{\New}[1]{#1}
\newcommand{\Old}[1]{}
\makeatletter
\def\blfootnote{\xdef\@thefnmark{}\@footnotetext}
\makeatother

\begin{document}

\title{The special Brauer group and twisted Picard varieties}

\author[D.\ Huybrechts, D.\ Mattei]{Daniel Huybrechts and Dominique Mattei}

\address{Mathematisches Institut \& Hausdorff Center for Mathematics,
Universit{\"a}t Bonn, Endenicher Allee 60, 53115 Bonn, Germany}
\email{huybrech@math.uni-bonn.de}

\address{Institute of Algebraic Geometry, Leibniz University Hannover, Welfengarten 1, 30167 Hannover, Germany}
\email{mattei@math.uni-hannover.de}

\begin{abstract} \noindent We generalise the notion of the Tate--{\v{S}}afarevi{\v{c}} group $\Sha(S/\PP^1)$ of an elliptic K3 surface $S\to \PP^1$ with a section
to $\Sha(S,h)$ of a K3 surface  $S$ endowed with a linear system $|h|$. 
The construction, which uses Grothendieck's special Brauer group, provides
an efficient way to deal with moduli spaces of twisted sheaves supported on curves in $|h|$.

 \vspace{-2mm}
\end{abstract}

\maketitle
\blfootnote{The authors are supported by the ERC Synergy Grant HyperK (ID 854361).}

\section{Introduction} 
Let $S_0\to\PP^1$ be an elliptic K3 surface with a section. 
Another elliptic K3 surface $S\to\PP^1$ (without a section) is called a twist of $S_0$
if its Jacobian fibration is $S_0\to \PP^1$, i.e.\ \New{$S_0\cong\overline\Pic^0(S/\PP^1)$}  relative over $\PP^1$. The twists of a fixed $S_0\to\PP^1$ are parametrised by the Tate--{\v{S}}afarevi{\v{c}} group $\Sha(S_0/\PP^1)$. According to a result originally due to \New{Grothendieck \cite{BrauerIII} and} Artin--Tate \cite{Tate},
this group is naturally isomorphic to the Brauer group, so  $\Sha(S_0/\PP^1)\cong\Br(S_0)$.

Changing perspective, every twist $S\to\PP^1$ of $S_0\to \PP^1$ can be viewed as a moduli space of rank one sheaves on the fibres of $S_0\to \PP^1$ twisted by some class $\alpha\in \Br(S_0)$. We will rephrase this by writing $S\to\PP^1$ 
as \New{ $\overline\Pic_\alpha^0(S_0/\PP^1)$} for some Brauer class $\alpha\in \Br(S_0)$.\smallskip

\subsection{} In this article we generalise the classical picture and consider twisted Picard varieties for arbitrary  generically smooth, complete linear systems $\kc\to|h|$
 of curves  contained in a K3 surface $S$. More specifically, this will lead us to consider twisted relative Jacobians
$\Pic_\alpha^d(\kc/|h|_{{\rm sm}})$ of the family $\kc\to|h|_{{\rm sm}}$  of all smooth curves
in $|h|$, generalising the classical relative Picard varieties
$\Pic^d(\kc/|h|_{{\rm sm}})$. However, in general the twists are not indexed by elements
in the Brauer group  $\Br(S)$ but by elements in a certain extension of it. 

\begin{thm}\label{thm:1}
Consider a complete, generically smooth, linear system $\kc\to |h|$ on a K3 surface $S$. Then there exists a natural \New{finite cyclic extension $\Sha(S,h)$ (the Tate--{\v{S}}afarevi{\v{c}} group):}
$$\xymatrix{0\ar[r]&\ZZ/m\ZZ\ar[r]&\Sha(S,h)\ar[r]&\Br(S)\ar[r]&0}$$
\New{that parametrises (possibly non effectively) all $\Pic^0(\kc/|h|_{{\rm sm}})$-torsors 
 isomorphic to a twisted relative Jacobian $\Pic^0_\alpha(\kc/|h|_{{\rm sm}})$.\smallskip}

Furthermore, the product in $\Sha(S,h)$ corresponds to the products
of twisted Picard varieties viewed as torsors for $\Pic^0(\kc/|h|_{\rm sm})\to |h|_{\rm sm}$, see
Remarks \ref{rem:twistint} \& \ref{rem:groupstructure}.
\end{thm}

The integer  $m$ in the theorem is the divisibility of $h$ as an element in the lattice $\NS(S)$ and we will later see that also all $\Pic_\alpha^d(\kc/|h|_{{\rm sm}})$ for non-zero $d$ are taken care of by $\Sha(S,h)$, see Proposition \ref{prop:getall}.\smallskip

This new  Tate--{\v{S}}afarevi{\v{c}} group $\Sha(S,h)$ extends the notion of the classical Tate--{\v{S}}afarevi{\v{c}}
group of an elliptic K3 surface $S_0\to \PP^1$ with a section to an
arbitrary elliptic K3 surface $S\to \PP^1$. It turns out that 
there exists a natural isomorphism $\Sha(S,f)\cong\Sha(S_0/\PP^1)$
between the new Tate--{\v{S}}afarevi{\v{c}} group $\Sha(S,f)$ 
and the classical one $\Sha(S_0/\PP^1)$ of its Jacobian fibration
{$S_0=\overline\Pic^0(S/\PP^1)$}, see Section \ref{sec:classTS}.
\smallskip

\subsection{} 
Working with twisted sheaves poses a number of technical problems. Firstly,
in order to talk about twisted sheaves it is not enough to just fix a Brauer class
$\alpha\in \Br(S)$. A certain geometric realisation is needed. This could
be an Azumaya algebra, a gerbe, a Brauer--Severi variety, or a \v{C}ech cocycle.
Secondly, to deal with moduli spaces, certain numerical invariants, e.g.\ Chern classes or Mukai vectors, need to be fixed. For example, one cannot define,
without introducing a certain ambiguity, the degree of an
$\alpha$-twisted sheaf on a fibre of $S\to\PP^1$. 
{ As was shown by Lieblich \cite[Ch.\ 5]{LiebPhD}, the ambiguity can be lifted by passing to $\mu_n$-gerbes, which is roughly the same as lifting a Brauer
class to the special Brauer group. Working with the special Brauer group allows
us to deal with all Brauer classes at the same time.}
\smallskip

To address these issues we make use of Grothendieck's special Brauer group
$\SBr(S)$ which is a certain extension of $\Br(S)$ by $\NS(S)\otimes\QQ/\ZZ$,
cf.\ \cite{BrauerII}. However, it will turn out that it is more convenient to work with a smaller
subgroup $\SBro(S)\subset\SBr(S)$, the restricted special Brauer group, 
which is a natural extension of $\Br(S)$ by the discriminant group of $\NS(S)$:
$$\xymatrix{0\ar[r]&\NS(S)^\ast/\NS(S)\ar[r]&\SBro(S)\ar[r]&\Br(S)\ar[r]&0.}$$
The various Tate--{\v{S}}afarevi{\v{c}} groups $\Sha(S,h)$ of curve classes on $S$  are then constructed as certain quotients $\SBro(S)\twoheadrightarrow\Sha(S,h)$.\smallskip

\subsection{} To get an idea of the role of the special Brauer group, 
let us look at the case of a smooth projective integral curve $C$ over an arbitrary field \New{$k$}.
In this case, there exists a short exact sequence of the form
$$\xymatrix{0\ar[r]&\QQ/\ZZ\ar[r]&\SBr(C)\ar[r]&\Br(C)\ar[r]&0.}$$ In addition
to the classical Picard varieties $\Pic^d(C)$, $d\in \ZZ$, which are torsors for
$\Pic^0(C)$, one can define varieties
$\Pic_\alpha^d(C)$ for any class $\alpha\in \SBr(C)$
and any rational number $d\in \QQ$.  If not empty, they are again torsors for $\Pic^0(C)$. Furthermore, \New{if $\alpha\in \SBr(C)$ is contained in the subgroup
$\QQ/\ZZ\subset\SBr(C)$, then $\Pic_\alpha^d(C)$ is a trivial torsor. The
converse holds if there exists a rational point $x\in C(k)$ such that the restriction $\alpha_x\in\Br(k(x)\cong k)$ is trivial. See the discussion in Section \ref{sec:modulicurve} for further details.}

\subsection{} Our discussion also sheds a new light on work of Markman \cite{Mark}. {Building upon his work}, we will prove the following result
 in Section \ref{sec:Mark}.

\begin{thm}\label{thm:main2}
Let $X\to \PP^n$ be Lagrangian fibration of a 
projective hyperk\"ahler manifold \New{of Picard rank two} of ${\rm K3}^{[n]}$-type. Then there exists
a K3 surface $S$, a complete linear system $\kc\to |h|$ on $S$, and a class
$\alpha\in\Sha(S,h)$ such that $X$ is birational to $\Pic^0_\alpha(\kc/|h|_{{\rm sm}})$
relative over $|h|\cong\PP^n$.
\end{thm}

 \New{The assumption on the Picard number can be weakened to a genericity assumption, cf.\ \S\! \ref{sec:Mark}}. If $X$ is not assumed projective, the K3 surface $S$ may be non-projective and for the class $\alpha$ one may have to use an analytic version of the special Brauer group.

\vskip0.5cm
\noindent
{\bf Conventions:}
In most of this article, we deliberately restrict to complex projective K3 surfaces, but there are interesting questions to explore for more general types of varieties as well as in more arithmetic settings. In our discussion it will sometimes be convenient
to deal with the scheme-theoretic generic curve in a linear system on a complex
projective K3 surface, in which case we work with curves over function fields
$\CC(t_1,\ldots,t_n)$.

For simplicity we will always assume that the complete linear system $\kc\to|h|$ parametrises at least one smooth curve or, equivalently, that the generic fibre is smooth. The cases of interest to us are  ample complete linear systems
and elliptic pencils.

For an Azumaya algebra $\ka$ we call $d(\ka)\coloneqq\sqrt{\rk(\ka)}$ the degree  of $\ka$.
\medskip

\noindent
{\bf Acknowledgements:} {We wish to thank Nick Addington for inspiration 
and Reinder Meinsma for helpful discussions. We are also grateful to Asher Auel and Evgeny Shinder for constructive criticism of the first version of the paper, to Alexei Skorobogatov for help with the literature, and to the referee for a thorough reading  and  the many pertinent comments.} The first author gratefully acknowledges the hospitality of the ITS-ETH Zurich during his stay in the spring of 2023.

\section{Grothendieck's special Brauer group}
The special Brauer group introduced by Grothendieck in \cite[Rem.\ 3.9]{BrauerII} is an extension of the classical Brauer group. In this section, we recall its definition and its cohomological description and explain how to use it to define Chern characters of twisted sheaves.
In the next section, we will then explain in what sense it is better suited to
study moduli spaces of twisted sheaves.

\subsection{Hodge theory} Assume $X$ is a smooth complex projective variety. We
are mainly interested in the case of K3 surfaces, but ultimately we will want to apply 
everything also to projective hyperk\"ahler manifolds.

By $\NS(X)$ we denote the N\'eron--Severi group of $X$
and  by $T(X)$ its transcendental lattice, i.e.\ the smallest saturated
sub-Hodge structure of $H^2(X,\ZZ)$ with $H^{2,0}(X)\subset T(X)\otimes \CC$.
The inclusion $\NS(X)\oplus T(X)\subset H^2(X,\ZZ)$ is rarely an equality but always of finite index. For simplicity we will ignore any torsion in $H^2(X,\ZZ)$.
Next, we define $T'(X)$ by the short exact sequence
\begin{equation}\label{eqn:NST'}
\xymatrix{0\ar[r]&\NS(X)\ar[r]& H^2(X,\ZZ)\ar[r]&T'(X)\ar[r]&0.}
\end{equation}
Then the composition $T(X)\subset H^2(X,\ZZ)\twoheadrightarrow T'(X)$ realizes
the transcendental lattice as a subgroup $T(X)\subset T'(X)$ of finite index. The quotient is a finite group,  which we will denote $$A(X)\coloneqq T'(X)/T(X).$$

For a surface $S$, the unimodular intersection form provides a natural identification $T'(S)\cong T(S)^\ast$ and $A(S)$ is nothing but the discriminant group of the transcendental lattice  or, equivalently, of the N\'eron--Severi lattice: $$A(S)  \cong T(S)^\ast/T(S)\cong \NS(S)^\ast/\NS(S).$$

Tensoring the exact sequence $0\to T(X)\to T'(X)\to A(X)\to 0$ with $\QQ/\ZZ$ 
induces an exact sequence, cf.\ \cite[Sec.\ 5.3]{CTS} or \cite[(9) \& Cor.\ 1.5]{GvSk}:
\begin{equation}\label{eqn:ATTses}
\xymatrix{0\ar[r]&A(X)\ar[r]&T(X)\otimes\QQ/\ZZ\ar[r]&T'(X)\otimes\QQ/\ZZ\ar[r]&0,}
\end{equation}
where we identified ${\rm Tor}_1(A(X),\QQ/\ZZ)$ with $A(X)$. The inclusion
is  explicitly described by first lifting an element in $A(X)$ to a class in $T'(X)$,
which is unique up to elements in $T(X)$, and then viewing it as an element in $T(X)\otimes \QQ=T'(X)\otimes\QQ$.

\begin{remark} Under our assumptions, a result of Gabber and de Jong, cf.\ \cite[Ch.\ 4]{CTS}, shows that the Brauer group equals the cohomological Brauer group, so $\Br(X)\cong H^2_{\text{\'et}}(X,\GG_m)$. The latter group can be identified
with the torsion subgroup $H^2(X,\ko_X^\ast)_{\rm tor}\subset H^2(X,\ko_X^\ast)$,
using the analytic topology. Furthermore, the exponential sequence provides an exact sequence
$$\xymatrix@C=18pt{0\ar[r]&\NS(X)\ar[r]&H^2(X,\ZZ)\ar[r]&H^2(X,\ko_X)\ar[r]&H^2(X,\ko^\ast_X)\ar[r]&H^3(X,\ZZ)\ar[r]&}$$
and in particular both groups $T(X)\subset T'(X)$ can be viewed as subgroups of $H^2(X,\ko_X)$. 

If $H^3(X,\ZZ)=0$, then 
\begin{equation}\label{eqn:BrT}
\Br(X)\cong H^2(X,\ko_X^\ast)_{\rm tor}\cong \left(H^2(X,\ko_X)/T'(X)\right)_{\rm tor}\cong T'(X)\otimes\QQ/\ZZ,
\end{equation} 
which for a surface $S$, using $T(S)^\ast\cong T'(S)$, is often written as $\Br(S)\cong\Hom (T(S),\QQ/\ZZ)$.

Without any assumption on $H^3(X,\ZZ)$, the isomorphism (\ref{eqn:BrT}) only describes the divisible part of $\Br(X)$.
\end{remark}

\subsection{Special Brauer group} In the very last remark of \cite{BrauerII}, Grothendieck introduced the  {special Brauer group} $\SBr(X)$ of a scheme $X$. It seems that this group has not been used much explicitly over the last fifty years, mainly because of its purely topological character. \New{However, it reflects a point of view that is central for Lieblich's treatment of twisted sheaves, see e.g.\ \cite{LiebPhD}, and is exactly what is needed for our purposes.}  The definition is similar to the 
definition of the Brauer group $\Br(X)$ as the group of Morita equivalence classes of Azumaya algebras.

\begin{definition}
The \emph{special Brauer group} $\SBr(X)$ of a scheme $X$ is the group of equivalence classes of Azumaya algebras $\ka$ with  respect to the equivalence relation generated by $\ka\sim \ka\otimes\kend(F)$ with the locally free sheaf $F$ required to
have trivial determinant $\det(F)\cong\ko_X$.
\end{definition}

\begin{remark}
(i) In order to ensure that with this definition the tensor product
$\ka\otimes\ka'$ defines a group structure on $\SBr(X)$ one needs
to use the fact that any Azumaya algebra $\ka$ has trivial determinant.\smallskip

(ii) Variants of the above definition exist. For example, instead of requiring $\det(F)$ to be trivial one can ask it to be only torsion or algebraically (or cohomologically)  trivial. It turns out that all these conditions eventually lead to the same group.
\end{remark}

By construction, $\SBr(X)$ naturally surjects onto $\Br(X)$ and the kernel has been determined in \cite[Rem.\ 3.9]{BrauerII}: There exists a short exact sequence
\begin{equation}\label{eqn:SBrBr}
\xymatrix{0\ar[r]&\Pic(X)\otimes\QQ/\ZZ\ar[r]&\SBr(X)\ar[r]&\Br(X)\ar[r]&0.}
\end{equation}

As explained in \cite[Rem.\ 3.9]{BrauerII}, the cohomological version of (\ref{eqn:SBrBr}) is obtained as the limit
of the exact sequences
$$\xymatrix{0\ar[r]&\Pic(X)/\ell^n\cdot\Pic(X)\ar[r]&H^2(X,\mu_{\ell^n})\ar@{->>}[r]&H^2_{\text{\'et}}(X,\GG_m)[\ell^n]\ar[r]&0}$$ induced by the Kummer sequence. In particular, whenever $\Br(X)\congpf H^2_\text{\'et}(X,\GG_m)$,
then $$\SBr(X)[\ell^\infty]\cong H^2(X,\QQ_\ell/\ZZ_\ell(1))={\lim_{\to}} H^2(X,\mu_{\ell^n}).$$ Since we assume $X$ to be a smooth complex projective variety, we can use singular cohomology to describe the situation. One finds $$\Pic(X)\otimes\QQ/\ZZ\cong \NS(X)\otimes\QQ/\ZZ\cong(\QQ/\ZZ)^{\oplus\rho(X)}$$ and
\begin{equation}\label{eqn:SBrcoh}
\SBr(X)\cong H^2(X,\QQ/\ZZ)\cong (\QQ/\ZZ)^{\oplus b_2(X)}\oplus H^3(X,\ZZ)_{\text{tors}}.
\end{equation}
If $H^3(X,\ZZ)_{\text{tors}}=0$, the sequence (\ref{eqn:SBrBr}) is obtained from (\ref{eqn:NST'}) by tensoring with $\QQ/\ZZ$:
\begin{equation}\label{eqn:sesSBrBr}
\xymatrix@R=6pt@C=12pt{0\ar[r]&\Pic(X)\otimes\QQ/\ZZ\ar[r]&\SBr(X)\ar[r]&\Br(X)\ar[r]&0\\&\cong \NS(X)\otimes\QQ/\ZZ&\cong H^2(X,\ZZ)\otimes\QQ/\ZZ&~~~\cong T'(X)\otimes \QQ/\ZZ.&}
\end{equation}

{\begin{remark}\label{rem:gerbes}
The Brauer group $\Br(X)\cong H^2(X,\GG_m)$ can also be viewed as the set of isomorphism classes of $\GG_m$-gerbes on $X$. Similarly, $H^2(X,\mu_n)$ is the
set of isomorphism classes of $\mu_n$-gerbes on $X$. From this perspective,
$\SBr(X)$ is the set of all $\mu_n$-gerbes on $X$, 
where a $\mu_n$-gerbe is identified with the naturally associated
$\mu_{nk}$-gerbe under $\mu_n\subset\mu_{nk}$.
\end{remark}}

\begin{remark}
(i) The injection in (\ref{eqn:sesSBrBr}) is geometrically realised by sending $(1/r)\cdot L$ with $L\in \Pic(X)$
to the Azumaya algebra given by $\kend(F)$, where $F$ is any vector bundle of rank $r$ and determinant $L$. Use
$\kend(F_1)\otimes\kend(F_1^\ast\otimes F_2)\cong \kend(F_2)\otimes\kend(F_1\otimes F_1^\ast)$ to see that this is independent of the choice of $F$. \smallskip

(ii) {This description fits with the interpretation of $\SBr(X)$ as parametrising $\mu_n$-gerbes. The $\mu_n$-gerbe associated with $\kend(F)$, i.e.\
its image under $H^1(X,{\rm PGl}(r))\to H^2(X,\mu_n)$, is the image
of $L\in \Pic(X)=H^1(X,\GG_m)$ under the boundary map
$\delta_r\colon H^1(X,\GG_m)\to H^2(X,\mu_r)$ induced by the Kummer
sequence.}\smallskip

\New{(iii)} As for the classical Brauer group, one checks that 
\New{if $\alpha\in \SBr(X)$ is represented by an Azumaya algebra $\ka$, then the order
of $\alpha$ as an element in the group $\SBr(C)$ divides $d(\ka)=\sqrt{\rk(\ka)}$.}
For example, if $\ka=\kend(F)$ with
$\rk(F)=r$, then $\alpha^r$ is realised by the endomorphism bundle 
of  $F^{\otimes  r}\otimes\det(F)^\ast$ which has trivial determinant.
\end{remark}
The cohomological description (\ref{eqn:SBrcoh}) of the special Brauer group  reveals that $\SBr(X)$ is a purely topological invariant of $X$, unlike the standard version $\Br(X)$, which depends on the Picard number and thus may change under deformations. The exact sequence (\ref{eqn:sesSBrBr})
provides a geometric interpretation for the link between jumps of the Picard number and drops of the rank of the Brauer group.

Note that both sequences, (\ref{eqn:SBrBr}) and (\ref{eqn:sesSBrBr}), actually split, but only non-canonically.

\begin{remark}\label{rem:Bfield1}
For a K3 surface $S$, it is common to call a
lift of a class $\alpha\in \Br(S)$ to an element in $H^2(S,\QQ)$
a \emph{B-field} lift of $\alpha$. Such a B-field lift then induces a lift
of $\alpha$ to a class in the special Brauer group $\SBr(S)$,  see Section \ref{sec:Bfield2} for more details.
\end{remark}


\begin{ex} For a smooth projective irreducible curve $C$ over \New{an arbitrary} field $k$, the
sequence (\ref{eqn:SBrBr}) becomes
$$\xymatrix@R=6pt{0\ar[r]&\QQ/\ZZ\ar[r]&\SBr(C)\ar[r]&\Br(C)\ar[r]&0.}$$
If $k$ is algebraically closed, then $\Br(C)$ is trivial and  $\QQ/\ZZ\cong\SBr(C)$.
Although we mostly consider complex curves, the  more general situation sheds light on the situation when applied to the scheme-theoretic generic fibre $\kc_\eta$ of a linear system $
\kc\to|h|$ on a surface.
\end{ex}

\subsection{Chern character}
Instead of looking at classes in the Brauer group $\Br(X)$ we now fix 
a class in the special Brauer group
$$\alpha\in\SBr(X).$$ Representing $\alpha$ by an Azumaya algebra
$\ka$, we can consider the abelian category $\Coh(X,\ka)$ (and its derived category). The same class $\alpha $ is also represented by any Azumaya algebra of
the form $\ka'\coloneq \ka\otimes \kend(F)$ with $F$ locally free and such that $\det(F)\cong \ko_X$. The two categories are equivalent to each other via 
\begin{equation}\label{eqn:AA'}
\Coh(X,\ka)\congpf \Coh(X,\ka'),~E\mapsto E\otimes F.
\end{equation} Since $\kend(F')\cong\kend(F)$ if and only if $F'\cong F\otimes L$
for some line bundle $L$ which, moreover, is torsion if $\det(F')\cong\det(F)$,
the equivalence (\ref{eqn:AA'})  is canonical up to equivalences given by
tensoring with torsion line bundles. In particular, there is a distinguished equivalence (\ref{eqn:AA'}) for K3 surfaces and hyperk\"ahler manifolds.

Since $X$ is assumed to be a smooth complex projective variety, we
may use singular cohomology and Hodge theory to fix cohomological invariants, but the following definitions can be adapted to other situations. As in \cite{HS2,HSeattle}, we consider the Chern character$$\ch_\ka\colon \Coh(X,\ka)\to H^*(X,\QQ), ~E\mapsto \ch_\ka(E)\coloneqq \sqrt{\ch(\ka)}^{-1} \cdot\ch(E).$$
Since we will mostly work with sheaves on curves (contained in K3 surfaces), 
multiplication with $\sqrt{\ch(\ka)}^{-1}$ is just 
division by the integer $d{(\ka)}$. For K3 surfaces, it is often more convenient to work with the Mukai vector $v_\ka(E)\coloneqq \ch_\ka(E)\cdot\td(S)^{1/2}$,
which for sheaves concentrated on curves in the K3 surface is the same as $\ch_\ka(E)$.
\smallskip

Note that on curves and surfaces $\ch_\ka$ commutes with the equivalence (\ref{eqn:AA'}). More
precisely, if $\det(F)\cong\ko$ and $\ka'=\ka\otimes\kend(F)$ then
$$\ch_{\ka'}(E\otimes F)=\ch_\ka(E),$$
for then ${\ch(\kend(F))}={\ch(F^\ast)\cdot\ch(F)}=\ch(F)^2$.


\begin{defprop}\label{defprop:Chern}
Let $\alpha\in \SBr(X)$ be a class in the special Brauer group on a variety of dimension at most two. Then the \emph{twisted Chern character} 
$$\ch_\alpha(E)\coloneqq\ch_\ka(E)$$
for $E\in \Coh(X,\alpha)\coloneqq\Coh(X,\ka)$ is independent of the choice
of the Azumaya algebra $\ka$ representing the class $\alpha\in \SBr(X)$.\qed
\end{defprop}

{\begin{remark}\label{rem:compLieb}
If $\alpha\in\SBr(X)$ is represented by an
Azumaya algebra $\ka\in H^1(X,{\rm PGl}(n))$, then $\Coh(X,\ka)$ is equivalent to
the category $\Coh(\km)_1$ of  sheaves of weight one on the associated
$\mu_n$-gerbe $\km_{\ka}$, cf.\ \cite{LiebDuke} and \cite[\S 3]{HSeattle} for 
further references. For two Azumaya algebras $\ka,\ka'\in H^1(X,{\rm PGl}(n))$ with isomorphic $\mu_n$-gerbes $\km_\ka\cong\km_{\ka'}$, i.e.\  inducing the same class in $H^2(X,\mu_n)$,
one could work with the Chern character $\ch_\km=\ch_{\km'}$ on the DM stacks $\km_\ka\cong\km_{\ka'}$ instead of $\ch_\alpha$. 

However, 
when passing from the $\mu_n$-gerbe $\km_\ka$ associated with an Azumaya algebra $\ka$  to the $\mu_{nr}$-gerbe $\km_{\ka'}$ associated with
$\ka'\coloneqq \ka\otimes\kend(E)$ for some locally free sheaf $E$ of rank $r$
and trivial determinant, the
situation becomes less clear. The induced map between the
DM stacks $\km_\ka\to\km_{\ka'}$ (over $X$), which is compatible with the natural inclusion
$\mu_n\hookrightarrow \mu_{nr}$, will, by functoriality of the Chern character,
pull-back $\ch_{\km_{\ka'}}(F\otimes E)$ to $\ch_{\km_\ka}(F)$. Here,
 $F\otimes E\in \Coh(X,\ka')\cong \Coh(\km_{\ka'})_1$ is the image of 
 $F\in \Coh(X,\ka)\cong \Coh(\km_\ka)_1$ under the equivalence $\Coh(X,\ka)\congpf \Coh(X,\ka')$, $F\mapsto F\otimes E$. Now, the fact that Proposition \ref{defprop:Chern} only holds in dimension $\leq 2$, shows that in general $\ch_{\km_\ka}(F)\ne\ch_\ka(F)$. Thus, expressing $\ch_\km(F)$ for $F\in \Coh(\km)$  in terms of classical Chern classes of $F$ viewed as coherent sheaf on $X$ seems to involve first passing to the `minimal'  gerbe of $\km$. There one would expect $\ch_\km(F)=\ch_\ka(F)$.
 
 This has bearings on the comparison between moduli spaces of sheaves over Azumaya algebras and moduli spaces of sheaves on gerbes, see Remark \ref{rem:gerbesmoduli}.
\end{remark}
}
\section{Moduli spaces of twisted sheaves}
Moduli spaces of (semi-)stable twisted sheaves have  been constructed by
Lieblich \cite{LiebDuke} and Yoshioka \cite{Yosh}, for the construction
from the point of view of modules over  Azumaya algebras see \cite{HS,Simpson}. 
Most of the classical arguments carry over. However, it is important to emphasise, that the construction of
the moduli space requires an additional choice, it is not enough to
fix a Brauer class $\alpha\in \Br(S)$. For example, one possibility is to choose  a \v{C}ech coycle $\{\alpha_{ijk}\}$  representing $\alpha$, which
allows one to use twisted sheaves. Another possibility is 
to fix a Brauer--Severi variety $P_\alpha\to S$ representing $\alpha$, which is used in \cite{Yosh}. For us it will be more convenient to view these moduli spaces as moduli
spaces of untwisted coherent sheaves $E$ \New{with a module structure } over an Azumaya
algebra $\ka$ representing a given Brauer class $\alpha$.

As we will only be interested in the birational type of moduli spaces, the difference
between stability and semi-stability and the choice of a polarisation, already needed to define stability, are all irrelevant. However, the question whether a moduli space is a coarse or a fine moduli space is important.
Over the stable locus, which is never empty in our setting, it is independent of the choice of a birational model.

\subsection{Moduli on curves}\label{sec:modulicurve}
Let us again first consider the case of a smooth projective irreducible curve $C$ over an arbitrary field $k$.  The discussion should be compared
to the one by Lieblich e.g.\ \cite[Sec.\ 3.2.2]{LiebComp} using the language of sheaves on
$\mu_n$-gerbes, see also the introduction of Section \ref{sec:ModSurf}.

\begin{definition}
For a class $\alpha\in \SBr(C)$ in the special Brauer group represented by an Azumaya algebra $\ka$ and $r\in \ZZ$, $d\in \QQ$ we denote by
$M_\alpha(r,d)$ the moduli space of semi-stable vector bundles $E\in \Coh(C,\ka)$ with $\ch_\ka(E)=(r,d)$. For the special case $r=1$ we use the notation
$$\Pic_\alpha^d(C)\coloneqq M_\alpha(1,d)$$ and call it the \emph{twisted Picard group} of
degree $d$.
\end{definition}

There is no ambiguity in this definition, as for any other Azumaya algebra $\ka'$ representing
$\alpha\in \SBr(C)$ there exists an equivalence $\Coh(C,\ka)\cong\Coh(C,\ka')$
unique up to tensoring with torsion line bundles which identifies
semi-stable vector bundles of Mukai vector $\ch_\ka=(r,d)$ in $\Coh(C,\ka)$ with those with $\ch_{\ka'}=(r,d)$ in $\Coh(C,\ka')$. However, the isomorphism between
two moduli spaces defined by means of different Azumaya algebras representing
the class in $\SBr(C)$ is natural only up to tensoring with torsion line bundles.

\New{Note that the degree  $d(\ka)=\sqrt{\rk(\ka)}$ of $\ka$ divides the rank
$\rk(E)$ (as a coherent sheaf) 
of any locally free sheaf $E\in \Coh(C,\ka)$. }
Therefore, the constant coefficient of $\ch_\ka(E)$ is indeed an integer,
which explains why we fixed $r\in \ZZ$. However, this is not true for the degree $d$, which is usually only contained in $(1/d{(\ka)})\ZZ$.

\begin{remark}\label{rem:range}
If $E\in \Coh(C,\ka)$ is (semi-)stable with 
$\ch_\ka(E)=(r,d)$, then also any line bundle twist
$E\otimes L\in \Coh(C,\ka)$ is (semi-)stable and
$\ch_\ka(E\otimes L)=(r,d+r\cdot \deg(L))$. So, if there exists a line bundle $L$ of degree one on $C$, then for fixed $r$ every moduli space $M_\alpha(r,d)$ is isomorphic to one with $d\in(1/d{(\ka)})\ZZ\cap[0,r)$. If we allow ourselves to twist only with line bundles of degree in $m\ZZ$, then $[0,r)$ has to be replaced by $[0,m\cdot r)$.
\end{remark}

\begin{remark}\label{rem:Brauertorsor}
 (i) All non-empty $\Pic^d_\alpha(C)$, $d\in \QQ$, $\alpha\in\SBr(C)$ are torsors for $\Pic^0(C)$ and this torsor structure does not depend on the choice of the Azumaya algebra either. If $k$ is algebraically closed, all non-empty $\Pic^d_\alpha(C)$ are trivial torsors and, therefore, non-naturally isomorphic to $\Pic^0(C)$.
 \smallskip
 
(ii) {Assume that there exists an $E\in \Coh(C,\ka)$ defining a point
in $\Pic_\alpha^d(C)$. Then $\rk(E)=d{(\ka)}$ and, therefore, the natural
injection $\ka\,\hookrightarrow \kend(E)$ is an isomorphism. For the latter, use that both sheaves are locally free of the same rank and with trivial determinant.} In particular, $$\alpha\in \New{\Pic}(C)\otimes\QQ/\ZZ\subset\SBr(C).$$
 {If, moreover, there exists a line bundle $L$ on $C$
  with $L^{\rk(E)}\cong\det(E)$,} then 
$\ka\cong \kend(E\otimes L^\ast)$ with $\det(E\otimes L^\ast)\cong \ko_X$ and,
hence,  $\alpha\in \SBr(C)$ is trivial. In particular,  for $k$ algebraically closed, the non-emptyness of $\Pic^d_\alpha(C)$
with $d$ an integer(!) implies that $\alpha\in \SBr(C)\cong\QQ/\ZZ$ is trivial,  because then {$\rk(E)\mid \deg(E)$.}

This is no longer true for $d\not\in\ZZ$. Indeed, any vector bundle $F$ of rank
$r$ can be viewed as a stable sheaf over $\ka=\kend(F)$
with $\ch_\ka(F)=1+\mu(F)$, i.e.\ $F\in \Pic^{\mu(F)}_\alpha(C)$
for $\alpha=[\ka]\in \Pic(C)\otimes\QQ/\ZZ\subset\SBr(C)$.
\end{remark}

 Let us now only fix an ordinary Brauer class $\alpha\in \Br(C)$. Then the moduli space $M_\alpha(r)$ of all stable twisted sheaves  of rank $r$ is still well defined
 but only locally of finite type. For example,  $$\Pic_\alpha(C)=M_\alpha(1)$$ parametrises all locally free   $E\in \Coh(C,\ka)$ of rank
 $\rk(E)=d{(\ka)}$. The definition does not depend on the choice
 of the Azumaya algebra $\ka$ representing $\alpha\in \Br(C)$.
 
Note that $ \Pic_\alpha(C)$ is a countable disjoint union of smooth projective varieties which are torsors for $\Pic^0(C)$. Moreover, it is induced
from $\Pic_{\tilde \alpha}^0(C)$, where $\tilde\alpha\in\SBr(C)$ is an
arbitrary lift of $\alpha$.
For later use we state this as the following.  
  
 \begin{prop}\label{prop:BrSBrPictorsor}
 Let $\tilde\alpha\in \SBr(C)$ be a lift of a class $\alpha\in \Br(C)$. 
Then the Picard scheme $\Pic_\alpha(C)$ is naturally a torsor for the group scheme $\Pic(C)= \bigsqcup \Pic^d(C)$, where the action is given by tensor product. More precisely, $$\Pic_{\alpha}(C)\cong (\Pic(C)\times \Pic_{\tilde\alpha}^0(C))/\Pic^0(C)$$ as torsors for $\Pic(C)$.\qed
\end{prop}

{We shall need the following general fact, cf.\ 
\cite[Sec.\ 5.1.3]{LiebPhD} or 
\cite[Prop.\ 3.2.2.6]{LiebComp} for the version for $\mu_n$-gerbes. For completeness sake, we include a sketch of the proof.}

\begin{prop}\label{prop:BrTorsCurves}
Let $C$ be a smooth projective curve over a field $k$. Then the map
\begin{equation}\label{eqn:HSBr}
\Br(C)\cong H^2_{\text{\rm \'et}}(C,\GG_m)\to H^1(k,\Pic(\bar C))
\end{equation}
induced by the Hochschild--Serre spectral sequence maps a Brauer class $\alpha\in \Br(C)$ to the 
class of the $\Pic(C)$-torsor $\Pic_\alpha(C)$.
\end{prop}

There is a version of this result for arbitrary schemes $X$, but the map 
(\ref{eqn:HSBr}) is then only defined on the kernel of $\Br(X)\to \Br(\bar X)$.

\begin{proof}
The map (\ref{eqn:HSBr}) is part of the commutative diagram
$$\xymatrix{\Br(C)\cong H_{\text{\rm \'et}}^2(C,\GG_m)\ar[r]& H^1(k,H_{\text{\rm \'et}}^1(\bar C,\GG_m))\\
H^1(C,{\rm PGL}(n))\ar[u]\ar[r]&H^0(k,H^1(\bar C,{\rm PGL}(n))),\ar[u]&}$$
where the right vertical map is the boundary map induced by the short exact sequence $$\xymatrix{0\ar[r]&\Pic(\bar C)\cong H^1(\bar C,\GG_m)\ar[r]&H^1(\bar C,{\rm GL}(n))\ar[r]&H^1(\bar C,{\rm PGL}(n))\ar[r]&0.}$$
Computing this boundary map explicitly shows
that $\ka\in H^1(C,{\rm PGL}(n))$ is mapped to the torsor which parametrises
all $E\in \Coh(C,\ka)$ with $\kend(E)\cong \ka$, which is nothing but $\Pic_\alpha(C)$.
\end{proof}

\begin{remark}\label{rem:expand}
 Let us elaborate on Remark \ref{rem:range}.

(i)  For simplicity we first assume that there exists a
line bundle $L_0$ of degree one on $C$ {and so, in particular,
$\Pic^1(C)(k)$ is not empty.}
 Consider a class $\alpha=(p/r)\cdot L_0\in \QQ/\ZZ\subset\SBr(C)$ with $p,r\in \ZZ$ coprime, so that $|\alpha|=r$, and pick a locally
free sheaf $F$  on $C$ with $\rk(F)=r$ and $\det(F)=L^p_0$. Then $\ka_F\coloneqq\kend(F)$ represents the class $\alpha\in \SBr(C)$ and locally free sheaves
 $E\in \Coh(C,\ka_F)$ with $\ch_{\ka_F}(E)=(1,d)$ are all of the form $E\cong F\otimes L$ for some line bundle $L$. 
 
 Since  $(1/r)\deg(F\otimes L)=(p/r)+\deg(L)$, this shows that $ d\equiv(p/r)$ modulo $\ZZ$, i.e.\ $\bar d=\alpha\in \QQ/\ZZ$. In other words,
 up to isomorphisms induced by multiplication with line bundles there
 exists only one non-empty twisted Picard variety $\Pic_\alpha^\alpha(C)$ (admittedly,  a somewhat confusing notation) which is furthermore non-naturally isomorphic to $\Pic^0(C)$.\smallskip
 

 (ii)  Let us now consider the case that the minimal positive degree of a line bundle
 $L_0$ on $C$ is $m$\New{, also called the index of $C$}. Similar arguments as above show the following: For a given $\alpha =(p/r)\cdot L_0\in\New{\Pic(C)}\otimes\QQ/\ZZ\subset\SBr(C)$, there exist, up to tensor products with line bundles,
 at most  $m$ twisted  Picard varieties $\Pic_\alpha^{d +i}(C)$, $i=0,\ldots,m-1$,
where $d=(m\cdot p/r)$, all torsors for  $\Pic^0(C)$. \New{In fact, there might be even fewer, as the index can be replaced by the period of the curve, i.e.\ the minimal $d$ such
that $\Pic^d(C)$ is a trivial torsor.}\smallskip

(iii) The two situations considered above will later be mixed as follows, \New{ cf.\ also
\cite[\S 5.1.2]{LiebPhD}}. We will consider curves $C\subset S$ in a complex
 projective K3 surface. In this case, there clearly exists a line bundle of degree one
 on each individual $C$, so that we can consider $\alpha=(p/r)\cdot L_0\in \SBr(C)$ with $\deg(L_0)=1$. However, in order to let $C$ vary in its linear system, we only allow twists by line bundles $L$ of degree $m$, where $m$ is determined by $(\NS(S).[C])=m\ZZ$. Thus, one considers
 $\Pic_\alpha^{d+i}(C)$ with $d=(p/r)$ and $i=0,\ldots,m-1$.
 \end{remark}

\subsection{Moduli on surfaces}\label{sec:ModSurf}
{Let us now turn to sheaves on \New{projective K3 surfaces over an algebraically closed field.}
Again, the discussion should be compared to the work of Lieblich e.g.\ in
\cite[Sec.\ 5.1]{LiebPhD}. In particular, he already introduced twisted Picard schemes for fibred surfaces. Note that in our setting, the curves are not necessarily the fibres
of a morphism but elements in a linear system.}\smallskip

According to Definition--Proposition \ref{defprop:Chern}, the 
twisted Chern character $\ch_\alpha(E)$ \New{is independent of the choice of the Azumaya algebra,} i.e.\ 
$\ch_\ka(E)=\ch_{\ka'}(E\otimes F)$ for $\ka'=\ka\otimes\kend(F)$ with $\det(F)\cong\ko_S$. Furthermore, the equivalence $\Coh(S,\ka)\cong\Coh(S,\ka')$
preserves stability with respect to a polarisation $\ko(1)$ which we will suppress
in the notation.\footnote{Note that in general only $\mu$-stability is preserved
under this equivalence. However, on surfaces and under the assumption that $F$ has trivial determinant, also Gieseker stability is preserved.}
In the case of K3 surfaces, it is more convenient to work
with the (twisted) Mukai vector $v_\alpha(E)=\ch(E)\cdot\td(S)^{1/2}$, which is
also well defined.

Hence, by \cite{HS,Simpson,Yosh} the moduli space $M_\alpha(v)$ of semi-stable $\alpha$-twisted sheaves with twisted Mukai vector $v_\alpha(E)=v$ is well defined as long as $\alpha\in\SBr(S)$ is fixed as a class in the special Brauer group.

{
\begin{remark}\label{rem:gerbesmoduli}
Once again, there should be a way to phrase everything in terms
of Lieblich's moduli spaces of sheaves on
$\mu_n$-gerbes, cf.\ \cite{LiebDuke,LiebComp}. In particular,
if a class $\alpha\in \SBr(X)$ is represented by a $\mu_n$-gerbe $\km$,
one should be able to compare moduli spaces of sheaves on $\km$
with a certain moduli space of sheaves on the $\mu_{nk}$-gerbe naturally
associated with $\km$ via the inclusion $\mu_n\subset \mu_{nk}$. However, comparing Chern characters, and so Hilbert polynomials, and stability is tricky in general,
see Remark \ref{rem:compLieb} and \cite[Lem.\ 2.3.2.8]{LiebDuke}.
\end{remark}}

We will be mainly interested in the case of sheaves $E$ on $S$ supported on curves $C\subset S$ in a (generically smooth) linear system $|h|$, in which case
there is no difference between the Mukai vector $v(E)$ and the Chern character $\ch(E)$. More precisely, the Mukai vector is in this case of the form $v(E)=(0,r\cdot h,s)$, where  $s=\chi(E)$. If $E$ is a line bundle on $C$, then $r=1$. In general, if $C$ is integral, $r$ is the rank of $E$
as a sheaf on $C$.\smallskip

We can now twist the situation with respect to a class
$\alpha\in\SBr(S)$ in the special Brauer group. Let us spell this out in detail.
Choosing an Azumaya algebra $\ka$
representing $\alpha$, we consider all sheaves $E\in \Coh(S,\alpha)=\Coh(S,\ka)$
with $v_\alpha(E)=v_\ka(E)=\ch_\ka(E)=(0,r\cdot h,s)$,  where  $s=\chi(E)$. 

\begin{definition} Consider a K3 surface $S$ with a generically smooth, complete linear system $\kc\to |h|$ and a special Brauer class $\alpha\in \SBr(S)$. Then we denote the moduli space $M_\alpha(0, h,s)$ of all semi-stable sheaves $E\in \Coh(S,\alpha)=\Coh(S,\ka)$ with $v_\alpha(E)=v_\ka(E)$ by 
$$\overline{\Pic}_\alpha^d(\kc/|h|)\coloneqq M_\alpha(0,h,s),$$
with $d\coloneqq s+g-1$,
and call it the \emph{compactified twisted relative Picard variety}.

\end{definition}

We will abuse the notation slightly and denote by
$$\New{\Pic_\alpha^d(\kc/|h|_{{\rm sm}})}\subset \overline{\Pic}_\alpha^d(\kc/|h|)$$
the open subset $\pi^{-1}(|h|_{{\rm sm}})$ of twisted sheaves concentrated on smooth curves $C\in |h|$. Here, 
\begin{equation}\label{eqn:piproj}
\pi\colon\overline\Pic_\alpha^d(\kc/|h|)\to|h|
\end{equation}
is the natural projection.
Note that stability depends on the choice of a polarisation and so $\overline{\Pic}_\alpha^d(\kc/|h|)$  depends on it as well. However, the open part 
\New{$\Pic_\alpha^d(\kc/|h|_{{\rm sm}})$} does not and as we are only interested in these moduli spaces up to birational isomorphism (over the \New{linear} system $|h|$), we can safely ignore the polarisation. 

\begin{remark}\label{rem:twistint}
 The fibre $\pi^{-1}(C)$ of (\ref{eqn:piproj})  over a smooth curve $C\in |h|$
consists of all (stable) locally free $\ka|_C$-sheaves of rank $d{(\ka)}$
and degree $d\cdot d{(\ka)}$. Thus, it is nothing but $\Pic^d_{\alpha|_C}(C)$ as
introduced in the previous section. In particular, if not empty, the fibre is naturally a torsor for $\Pic^0(C)$. {More globally, if not empty, 
$\Pic_\alpha^d(\kc/|h|)\to |h|_{{\rm sm}}$ is a torsor for the abelian group
scheme $\Pic^0(\kc/|h|_{{\rm sm}})\to |h|_{{\rm sm}}$.}

Note that, according to Remark \ref{rem:Brauertorsor},  if the fibre is non-empty and $d\in \ZZ$ is an integer, then $\alpha|_C=1\in \SBr(C)$.
%
\end{remark}

Tensor product with a line bundle $L$ on $S$ with $(L.h)=m$
defines  an isomorphism
$$
\overline\Pic^d_\alpha(\kc/|h|)\congpf \overline\Pic^{d+m}_\alpha(\kc/|h|)
~\text{  and }~
\Pic^d_\alpha(\kc/|h|_{{\rm sm}})\congpf \Pic^{d+m}_\alpha(\kc/|h|_{{\rm sm}}).$$

\Old{Tensor product with a line bundle $L$ on $S$ with $(L.h)=m$
defines  isomorphisms
$$M_\alpha(0,r\cdot h,s)\congpf M_\alpha(0,r\cdot h,s+r\cdot m)~\text{  and }~
\Pic^d_\alpha(\kc/|h|)\congpf \Pic^{d+m}_\alpha(\kc/|h|).$$}
As in the case of curves, $s$ and $d$ need not be integers. But note that 
 if for a fixed
class $\alpha\in\SBr(S)$ and two $d,d'\in \QQ$ the two relative twisted Picard varieties
$\Pic^d_\alpha(\kc/|h|_{{\rm sm}})$ and $\Pic_\alpha^{d'}(\kc/|h|_{{\rm sm}})$ are both not empty, then $d'=d+i $ for some integer $i$, see Remarks  \ref{rem:range} \& \ref{rem:expand}.

\begin{remark}\label{rem:nodist}
Using Remark \ref{rem:expand}, we find that up to tensoring with line bundles
on $S$ there are \New{at most} $m$ \New{$\alpha$-}twisted relative Picard schemes $\Pic_{\alpha}^{d+i}(\kc/|h|_{{\rm sm}})$
with $i=0,\ldots,m-1$. Here, as before, $m$ satisfies $m\ZZ=(\NS(S).h)$, i.e.\ it is the divisibility of $h$ as an element of the lattices $\NS(S)$,  and the rational number
$d=p/r$ is determined by writing $\alpha|_C=(p/r)\cdot L_0$ for some line bundle
$L_0$ of degree one on $C$ and with coprime $p$ and $r$.

In particular, unless $d=0$, there is no preferred choice for the degree $d+i$ that would work
well with the group structure of $\SBr(S)$. This observation will give rise to introducing the restricted special Brauer group in the next section.
\end{remark}

Let us conclude this section by explaining a relative version of Remark \ref{rem:Brauertorsor}.

\begin{remark}\label{rem:Brauertorsor2}
If instead of a class in $\SBr(S)$ we only fix a Brauer class $\alpha\in \Br(S)$, then
we define \Old{$$\Pic_\alpha(\kc/|h|)\to |h|$$}
$$\overline\Pic_\alpha(\kc/|h|)\to |h|$$
 as the moduli space of all (semi-)stable sheaves $E\in \Coh(S,\ka)$ with Mukai vector $\ch_\ka(E)=(0,h,\ast)$.
This is a countable disjoint union of projective schemes over $|h|$. The scheme-theoretic generic fibre is $\Pic_{\alpha|_{\kc_\eta}}(\kc_\eta)$, where $\kc_\eta$ is
the generic fibre of $\kc\to|h|$. 

Restricting to smooth curves in $|h|$, the scheme
$\Pic_\alpha(\kc/|h|_{{\rm sm}})\to |h|_{{\rm sm}}$ is a torsor for the 
countable union of projective group schemes
$\Pic(\kc/|h|_{{\rm sm}})=\bigsqcup\Pic^d(\kc/|h|_{{\rm sm}}) \to |h|_{{\rm sm}}$ 

\end{remark}
\section{The restricted special Brauer group}
From the perspective of moduli spaces of twisted sheaves  (or, rather, of modules over Azumaya algebras) on curves contained in
K3 surfaces, not all classes $\alpha\in \SBr(S)$ are relevant.
Only classes in a drastically smaller group give naturally rise to non-empty moduli spaces, cf.\ Remarks \ref{rem:Brauertorsor} \& \ref{rem:nodist}. This leads to the notion of the restricted special Brauer group.

\subsection{Restricted special Brauer group}
\New{Using the dual of the inclusion $\NS(S)\subset H^2(S,\ZZ)$ and the unimodularity of $H^2(S,\ZZ)$, one defines a natural map
\begin{equation}\label{eqn:SBR}
\SBr(S)\cong H^2(S,\ZZ)\otimes \QQ/\ZZ\to \NS(S)^\ast\otimes \QQ/\ZZ.
\end{equation}
The \emph{restricted special Brauer group} $\SBro(S)\subset \SBr(S)$ is then defined as the kernel of this map. Equivalently, $\SBro(S)\subset \SBr(S)$
is the subgroup that annihilates $\NS(S)$ under the natural 
pairing
\begin{equation}\label{eqn:Pairing}
\SBr(S)\times \NS(S)\cong H^2(S,\QQ/\ZZ)\times\NS(S)\to \QQ/\ZZ.
\end{equation}

}

\New{From the definition of the restricted special Brauer group $\SBro(S)$ we deduce the following commutative
diagram
\begin{equation}\label{eqn:sesASBro}
\xymatrix{0\ar[r]&A(S)\ar@{^(->}[d]\ar[r]&\SBro(S)\ar@{^(->}[d]\ar[r]&\Br(S)\ar@{=}[d]\ar[r]&0\\
0\ar[r]&\NS(S)\otimes\QQ/\ZZ\ar@{->>}[d]\ar[r]&\SBr(S)\ar[r]\ar@{->>}[d]&\Br(S)\ar[r]&0\\
&\NS(S)^\ast\otimes\QQ/\ZZ\ar@{=}[r]&\NS(S)^\ast\otimes\QQ/\ZZ.&&}
\end{equation}
It is not difficult to check that $\SBro(S)$ still surjects onto $\SBr(S)$.
}
%


\New{The subgroup $\SBro(S)$ parametrises all Brauer classes with respect to which every complete linear system can be twisted. However, once a generically smooth, complete linear system $|h|$ on $S$
is fixed, we replace $\SBro(S)$ by a larger subgroup
$$\SBro(S)\subset \SBr(S,h)\subset\SBr(S).$$

To define $\SBr(S,h)$, we consider again the pairing (\ref{eqn:Pairing})
and let $\SBr(S,h)\subset \SBr(S)$ be the annihilator of $h\in \NS(S)$, i.e.\
the set of all classes $\alpha\in \SBro(S)\cong H^2(S,\ZZ/\QQ)$ with
$(\alpha.h)=0$ in $\QQ/\ZZ$.
Analogously, we define $A(S,h)\subset \NS(S)\otimes \QQ/\ZZ$
as the annihilator of $h\in \NS(S)$ with respect to the pairing
$$\left(\NS(S)\otimes \QQ/\ZZ\right)\times \NS(S)\to\QQ/\ZZ.$$

Altogether, we have the commutative diagram of short exact sequences
that completes (\ref{eqn:sesASBro}):
$$\xymatrix{0\ar[r]&A(S)\ar@{^(->}[d]\ar[r]&\SBro(S)\ar@{^(->}[d]\ar[r]&\Br(S)\ar@{=}[d]\ar[r]&0\\
0\ar[r]&A(S,h)\ar@{^(->}[d]\ar[r]&\SBr(S,h)\ar@{^(->}[d]\ar[r]&\Br(S)\ar@{=}[d]\ar[r]&0\\
0\ar[r]&\NS(S)\otimes\QQ/\ZZ\ar[r]&\SBr(S)\ar[r]&\Br(S)\ar[r]&0.}$$
Again, the surjectivities on the right are straightforward to verify.}

\begin{remark}\label{rem:SBrodef}
Viewing $\SBr(S)$ as the group that parametrises all $\mu_n$-gerbes, $n\in \ZZ$, on $S$,  see Remark \ref{rem:gerbes}, we find that $\SBro(S)\subset \SBr(S)$ is the subgroup of all $\mu_n$-gerbes that become trivial on all \New{smooth, integral} curves contained  $C\subset S$.

Indeed, $\NS(S)$ is generated by classes of smooth integral
curves and for a class $\gamma\in H^2(S,\QQ)$ the condition
$(\gamma.[C])\in \ZZ$ is  equivalent to $\gamma|_C\in H^2(C,\QQ)$
being contained in $H^2(C,\ZZ)$. In other words, the special Brauer class
induced by $\gamma$ is contained in the kernel of
$$\SBr(S)\to\SBr(C)\cong H^2(C,\QQ/\ZZ).$$

Similarly, $\SBr(S,h)$ is the set of classes
$\alpha\in \SBr(S)$ with $\alpha|_C\in \SBr(C)\cong H^2(C,\QQ/\ZZ)$ being trivial for
all smooth curves $C\in |h|$.
\end{remark}
\subsection{Restricted Brauer group via moduli spaces} The restricted special Brauer group can be alternatively characterised by the non-emptiness of moduli spaces.

\New{\begin{prop}
Assume $\kc\to|h|$ is a {generically smooth}, complete linear system.
 Then 
$\SBr(S,h)\subset\SBr(S)$ is the subgroup of all special Brauer
classes such that $\Pic^0_\alpha(\kc/|h|_{\rm sm})$ is not empty.
\end{prop}}

\begin{proof}
\New{Let $\alpha\in \SBr(S,h)$ and pick any smooth curve $C\in |h|$.} By Remark \ref{rem:SBrodef}, the class $\alpha|_C\in \SBr(C)$ is trivial. Hence, $\alpha|_C$ can be represented by $\ko_C$ and any degree zero line bundle, e.g.\ $\ko_C$ itself, defines a point in $\Pic^0_{\alpha|_C}(C)$. In particular, $\Pic_\alpha^0(\kc/|h|_{\rm sm})\ne\emptyset$.

Conversely, if $\Pic^0_\alpha(\kc/|h|_{\rm sm})\ne\emptyset$ for a fixed class
$\alpha\in\SBr(S)$, 
then according to Remark \ref{rem:Brauertorsor}, $\alpha|_C\in \SBr(C)$ is trivial for all smooth curves
$C\in|h|$ and hence $\alpha\in \SBro(S\New{,h})$.
\end{proof}

It is important to emphasise that although $\alpha|_C\in \SBr(C)$ is trivial
for a class $\alpha\in \SBr(S,h)$ and any smooth curve $C\in|h|$,
the restriction $\alpha|_{\kc_\eta}\in \SBr(\kc_\eta)$ to the generic
fibre of the complete linear system $\kc\to|h|$ is trivial if and only if $\alpha$ is of the form
$(1/r)\cdot L\in \NS(S)\otimes\QQ/\ZZ$ with $\deg(L|_C)=0$. So, in particular, the associated class $\bar \alpha\in\Br(S)$ would be trivial in this case, \New{for the restriction map $\Br(S)\,\hookrightarrow\Br(\kc_\eta)$ is injective.}

\subsection{Generalised Tate--{\v{S}}afarevi{\v{c}} group}\label{sec:mainproof}
It turns out that for a given complete linear system  $\kc\to|h|$ the restricted
special Brauer group \New{$\SBr(S,h)$} does not parametrise the moduli spaces $\Pic^0_\alpha(\kc/|h|_{\rm sm })$ effectively. We will show that the map that associates with a class
$\alpha\in\SBr(S,h)$ the moduli space $\Pic^0_\alpha(\kc/|h|_{\rm sm})$ factorises via a certain quotient of $\SBro(S)\twoheadrightarrow \Sha(S,h)$.

For this purpose, we consider for any $h\in \NS(S)$ the natural map
$$\zeta_h\colon 
A(S,h)\New{\twoheadrightarrow }\ZZ/(\NS(S).h),~\varphi\mapsto \varphi(h).$$
\New{By the very definition of $A(S,h)$, this map is indeed surjective.
Note that for $h$ primitive also the restriction of $\zeta_h$ to $A(S)\subset A(S,h)$ is surjective, which has the consequence that
in the definition of the generalised Tate--{\v{S}}afarevi{\v{c}} group $\Sha(S,h)$, one could as well use the smaller  $\SBro(S)$. However, for $h=kh_0$ with $h_0$ primitive, the subgroup $\zeta_h(A(S))$  is of index $k$.}


\begin{definition} Let $\kc\to|h|$ be a generically smooth, complete linear system on a K3 surface $S$.
Then the \emph{Tate--{\v{S}}afarevi{\v{c}} group} of $(S,h)$ is defined
as $$\Sha(S,h)\coloneqq \SBro(S\New{,h})/\ker(\zeta_h)$$ and we denote the projection
by
\begin{equation}\label{eqn:SBroSha}
\xi_h\colon \SBro(S\New{,h})\twoheadrightarrow \Sha(S,h).
\end{equation}
\end{definition}

Thus,  there exists a short exact sequence
\begin{equation}\label{eqn:Shases}
\xymatrix@C=15pt{0\ar[r]&\ZZ/(\NS(S).h)\ar[r]&\Sha(S,h)\ar[r]&\Br(S)\ar[r]&0,}
\end{equation}
which together with (\ref{eqn:sesASBro}) is part of the commutative diagram


\begin{equation}\label{eqn:CDShah}
\xymatrix{
0\ar[r]&\NS(S)\otimes\QQ/\ZZ\ar[r]&\SBr(S)\ar[r]&\Br(S)\ar[r]&0\\
0\ar[r]&A(S,h)\ar@{^(->}[u]\ar@{->>}[d]_{\zeta_h}\ar[r]&\SBr(S,h)\ar@{^(->}[u]\ar[r]\ar@{->>}[d]_{\xi_h}&\Br(S)\ar@{=}[u]\ar@{=}[d]\ar[r]&0\\
0\ar[r]&\ZZ/(\NS(S).h)\ar[r]&\Sha(S,h)\ar[r]&\Br(S)\ar[r]&0.}
\end{equation}
\smallskip

\begin{ex}\label{exa:divone}
The natural projection is an isomorphism
$\Sha(S,h)\congpf \Br(S)$ if and only if $h\in\NS(S)$ has divisibility one, i.e.\
$(\NS(S).h)=\ZZ$. This is reminiscent of the case of elliptic K3 surfaces
and will be discussed in detail further below.
\end{ex}

\begin{prop}\label{prop:comptwo} Consider a  generically smooth, complete linear system $\kc\to|h|$. Then, for any two classes $\alpha_1,\alpha_2\in \SBr(S,h)$ with $\xi_h(\alpha_1)=\xi_h(\alpha_2)$, the two moduli spaces $\Pic_{\alpha_1}^0(\kc/|h|_{\rm sm})$  and $\Pic_{\alpha_2}^0(\kc/|h|_{\rm sm })$ are \New{naturally isomorphic torsors
for $\Pic^0(\kc/|h|_{\rm sm})$.} 
\end{prop}

\begin{proof}
Write $\alpha_2=\alpha_1\cdot\gamma$ with $\gamma=(1/r)\cdot L\in \NS(S)\otimes\QQ/\ZZ$. If $\gamma\in A(S,h)\subset \SBr(S,h)$, then $r\mid(L.h)$,
and under the stronger assumption $\gamma\in \ker(\zeta_h)\subset A(S,h)$, we
find in addition a line bundle $L'\in \NS(S)$ such that $(L'.h)+(1/r)(L.h)=0$.
Thus, we can assume that $\gamma$ is of the form $(1/r)\cdot L$ with $(L.h)=0$.

Let now $F$ be a locally free sheaf on $S$ with $\det(F)\cong L$ and $\rk(F)=r$\New{, e.g.\
$F=L\oplus\ko_S^{\oplus r-1}$.}
Then $$\New{\Pic_{\alpha_1}^0(\kc/|h|_{\rm sm})\congpf \Pic_{\alpha_2}^0(\kc/|h|_{\rm sm})}, ~E\mapsto E\otimes F$$ is an isomorphism of torsors for \New{$\Pic^0(\kc/|h|_{\rm sm})$.}
\end{proof}

\begin{remark}\label{rem:ShahBr} Consider a generically smooth, complete linear system $\kc\to |h|$.
Then, without introducing any ambiguity, the proposition allows us to speak of the moduli spaces $\Pic^0_\alpha(\kc/|h|_{\rm sm})$ for any class  $\alpha\in \Sha(S,h)$
 in the Tate--{\v{S}}afarevi{\v{c}} group of $(S,h)$.
\end{remark}

\begin{remark}\label{rem:groupstructure} 
 \New{As already hinted at in the last proof, the group structure of $\Sha(S,h)$ as a quotient of $\SBr(S,h)$}  is compatible with taking twisted Picard varieties.
More precisely, for two classes $\alpha_1,\alpha_2$ the torsor
 \New{$\Pic_{\alpha_1\alpha_2}^0(\kc/|h|_{\rm sm})$ for $\Pic^0(\kc/|h|_{\rm sm})$,}
 is the quotient
 $$\Pic^0_{(\alpha_1\alpha_2)}(\kc/|h|_{\rm sm})\cong\left(\Pic^0_{\alpha_1}(\kc/|h|_{\rm sm})\times_{|h|} \Pic^0_{\alpha_2}(\kc/|h|_{\rm sm})\right)/\Pic^0(\kc/|h|_{\rm sm}),$$
 where the isomorphism is given by tensor product.
 \end{remark}
 
\subsection{All twisted Picard varieties} Although we only consider classes in the restricted special Brauer group $\SBro(S\New{,h})\subset\SBr(S)$ and only consider $d=0$, we still get all twisted Picard varieties of arbitrary degree 
for  $\kc\to|h|_{\rm sm}$. This is the next result which
morally is  a consequence of the surjectivity $\SBr(S,h)\twoheadrightarrow \Br(S)$.

\begin{prop}\label{prop:getall}
Consider a {generically smooth}, complete linear system $\kc\to |h|$. Fix $d\in \QQ$ and
$\alpha\in \SBr(S\New{,h})$, such that $\Pic_\alpha^d(\kc/|h|_{\rm sm})$ is non-empty.
 Then there exists a class $\alpha_0\in \SBro(S\New{,h})$ (or $\alpha_0\in\Sha(S,h)$) and a natural isomorphism of $\Pic^0(\kc/|h|_{\rm sm})$-torsors
 $$\Pic^0_{\alpha_0}(\kc/|h|_{\rm sm})\cong \Pic_\alpha^d(\kc/|h|_{\rm sm}).$$
\end{prop}

\begin{proof} 
\New{To ease the notation, we will 
simply write $\Pic_\alpha^d$ instead of $\Pic_{\alpha}^d(\kc/|h|_{\rm sm})$, etc.}

As a warmup, let us first discuss the special case that
$\alpha=(1/r)\cdot L\in \NS(S)\otimes\QQ/\ZZ\subset\SBr(S)$ and $d=(L.h)/r$.
Then we let  $\alpha_0$ be the trivial class and define $$\Pic_{\alpha_0}^0=\Pic^0\congpf\Pic^d_\alpha,~E\mapsto E\otimes F.$$
Here, $F$ is a locally free sheaf on $S$ of rank $r$ and determinant $L$,
hence $\alpha=[\kend(F)]$, and $E$ is a degree zero line bundle on some smooth 
$C\in |h|$. In particular, the tensor product with $F$ is actually the tensor product with $F|_C$ on $C$, cf.\ Remarks \ref{rem:Brauertorsor} \& \ref{rem:expand}.
The existence of the above isomorphism is of course equivalent to
saying that the $\Pic^0$-torsor $\Pic_{\alpha}^d$ is trivial, \New{as it admits the section provided by the restriction of $F$}.

This proves the result for $\alpha\in \NS(S)\otimes\QQ/\ZZ$, but only
for one particular $d$. To obtain all $d+i$, one has to work with
non-trivial classes $\alpha_0\in A(S\New{,h})$. To be precise,
fix a line bundle $L_0$ on $S$ with $(L_0.h)\ZZ=(\NS(S).h)$ and let
$\New{\alpha_0=(i/(L_0.h))\cdot L_0  \in A(S,h)}$. Pick a locally free sheaf $F_0$ with $\rk(F_0)=(L_0.h)$
and  $\det(F_0)=L_0^i$. Then for $\alpha_0\coloneqq[\kend(F_0)]$ one has $\Pic_{\alpha_0}^0\congpf\Pic_\alpha^{d+i}$ via $E\mapsto E\otimes F_0$. 
\New{Note that in this case, the torsor $\Pic_\alpha^{d+i}$ is not necessarily trivial anymore.}
\smallskip

For the general case, write any $\alpha\in \SBr(S)$ as a product
$\alpha=\alpha_0\cdot\alpha_1$ with $\alpha_0\in \SBro(S\New{,h})$ and $\alpha_1\in \NS(S)\otimes \QQ/\ZZ$. Once the decomposition is picked, we
first deal with $\Pic^d_\alpha(\kc/|h|_{\rm sm})$ for one choice of $d$ and then show how to modify the decomposition $\alpha=\alpha_0\cdot\alpha_1$ to $\alpha=(\alpha_0\cdot\gamma^{-1})\cdot(\gamma\cdot \alpha_1)$ by some
$\gamma\in A(S\New{,h})$ to obtain all $d+i$, cf.\ Remark \ref{rem:twistint}. 

For the fixed choice of $\alpha_0=[\ka_0]$ and $\alpha_1=[\kend(F)]$, one has
$$\Pic^0_{\alpha_0}\congpf\Pic_\alpha^d, ~E\mapsto E\otimes F$$ with
$d=\mu(F|_C)$. Now let $\gamma\in A(S\New{,h})$ be the class $(i/(L_0.h))\cdot L_0$ and choose $F_0$ as above. Then $\Pic_{\alpha_0\cdot\gamma^{-1}}^0\congpf
\Pic_\alpha^{d+i}$ via $E\mapsto E\otimes (F\otimes F_0)$ and $\alpha_0\cdot\gamma^{-1}$.
\end{proof}

\begin{remark}\label{rem:reiter}
\New{We reiterate, see beginning \S\! \ref{sec:mainproof}, that for $h$ primitive $\Sha(S,h)$ can also be defined as the quotient of the smaller
subgroup $\SBro(S)\subset\SBr(S,h)$ by
the kernel of the restriction $\zeta_h\colon A(S)\to \ZZ/(\NS(S).h)$. In particular, as can also be seen by going through the above proof, the $\alpha_0$ in the proposition can be chosen to be contained in $\SBro(S)$.

Note that for $h=k\cdot h_0$, one can view $\Sha(S, h_0)$ as a subgroup
of $\Sha(S,h)$ with a cyclic quotient of order $k$. This inclusion maps
$\Pic^1(\kc/|h_0|_{\rm sm})$ to  $\Pic^k(\kc/|h|_{\rm sm})$.}
\end{remark}

\begin{ex}\label{ex:genexa}
In particular, all the untwisted relative Picard varieties $\Pic^d(\kc/|h|_{\rm sm})$
are still accounted for, namely 
$$\Pic^d(\kc/|h|_{\rm sm})\cong\Pic_{\New{-}\bar d}^0(\kc/|h|_{\rm sm}).$$
{Here,  $\bar d\in \ZZ/(\NS(S).h)$} is defined as  the image
of $(d/(L_0.h))L_0\in A(S\New{,h})\subset \NS(S)\otimes\QQ/\ZZ$,
where $L_0\in\NS(S)$ satisfies $(L_0.h)\ZZ=(\NS(S).h)$. Hence,
 $\Pic^0_{\bar d}(\kc/|h|_{\rm sm})$ is a moduli space
in $\Coh(S,\kend(F))$ with $F$  any locally free sheaf of rank $(L_0.h)$ and determinant {$L_0^{-d}$.}
To conclude,  use that {$\deg(F|_C\otimes L)=-(L_0^{d}.C)+d \cdot \rk(F)=0$}
which leads to an isomorphism
 $$\xymatrix@R=12pt{\Pic^d(\kc/|h|_{\rm sm})\,\ar[r]^\sim& \Pic^0_{\New{-}\bar d}(\kc/|h|_{\rm sm})},~L\mapsto F|_C\otimes {L}.$$ \end{ex}


\New{
The discussion so far says that
\begin{equation}\label{eqn:Inj}
\Sha(S,h)\to\left\{\, \Pic^0(\kc/|h|_{\rm sm})\text{-torsors}\,\right\}/_\cong,~
\alpha\mapsto \Pic_\alpha^0(\kc/|h|_{\rm sm})
\end{equation}
is a group homomorphism whose image consists of all torsors of the form $\Pic_\alpha^d(\kc/|h|_{\rm sm})$. 
\smallskip

Ideally, one would like to prove the injectivity of (\ref{eqn:Inj}).\footnote{The
fact that we prove in the proposition below the injectivity only for classes of sufficiently high order 
is the reason for the `possibly non effectively' in Theorem \ref{thm:1}.}
For this, one would need to show that if for a class $\alpha\in \SBr(S,h)$ the torsor $\Pic_\alpha^0(\kc/|h|_{\rm sm})$ is trivial, then
$\alpha$ is contained in the subgroup
$\ker(\xi_h)=\ker(\zeta_h)$. Let us first explain the two problems  one 
has to overcome  and then explain in Proposition \ref{prop:inj} how to
deal with them under the assumption that the order of 
$\alpha\in\SBr(S,h)$ is sufficiently large.
\smallskip

(i) Assume $\alpha\in \SBro(S)$ such that $\Pic_\alpha^0(\kc/|h|_{\rm sm})$ is a trivial torsor. Then the pull-back $\bar\alpha|_\kc\in \Br(\kc)$ of $\bar\alpha\in \Br(S)$ to the smooth part $\pi\colon\kc\to|h|_{\rm sm}$ is contained in the kernel of the natural map
$\Br(\kc)\to H^1(|h|_{\rm sm},R^1\pi_\ast\GG_m)$. This kernel is a quotient
of $\Br(|h|_{\rm sm})$ which might be non-trivial. If $\bar\alpha|_\kc\in \Br(\kc)$ could be shown to be trivial, then the injectivity of $\Br(S)\,\hookrightarrow \Br(\kc)$ would prove that also $\bar\alpha\in \Br(S)$ is trivial.
\smallskip

(ii) Assuming that (i) has been carried out, then we know already that $\alpha$ is contained in $A(S,h)=\ker\left(\SBr(S,h)\to\Br(S)\right)$.
As we are only interested in the image of $\alpha$ in the group $\ZZ/(\NS(S).h)\subset
\Sha(S,h)$, we may assume that $\alpha$ is of the form $(d/(L_0.h)) L_0$,
where $L_0\in\NS(S)$ with $(L_0.h)\ZZ=(\NS(S).h)$. Then, $\Pic^0_\alpha(\kc/|h|_{\rm sm})$ is isomorphic to the untwisted torsor
$\Pic^d(\kc/|h|_{\rm sm})$, see Example \ref{ex:genexa} below for details.
Since we are assuming that $\Pic^0_\alpha(\kc/|h|_{\rm sm})$ is trivial,
both torsors, $\Pic^0_\alpha(\kc/|h|_{\rm sm})$ and, hence also, $\Pic^d(\kc/|h|_{\rm sm})$, admit sections. 

Suppose now that the section of $\Pic^d(\kc/|h|_{\rm sm})$ corresponds to
a line bundle $\kl$ on $\kc\to|h|_{\rm sm}$ (this is a priori obstructed
by a class in $\Br(|h|_{\rm sm})$.)
Then extending $\kl$ to the complete family $\bar\kc\to|h|$ and using that $\bar\kc\to S$ is a projective bundle,
we may write $\kl$ as a pull-back of a line bundle $L$ on $S$ up to twists
by the tautological bundle on $|h|$. But then $(L.h)=d$ is a multiple of $(L_0.h)$
and, therefore, the class of $\alpha$ in $\ZZ/(L_0.h)\ZZ$ is trivial.
\smallskip 

Summarising, in both steps, (i) and (ii), there is potentially an obstruction
in $\Br(|h|_{\rm sm})$ to carry out the program and prove the injectivity 
of (\ref{eqn:Inj}). Note that for elliptic K3 surfaces, i.e.\ $|h|$ is of dimension one, the Brauer group
 $\Br(|h|_{\rm sm})$ is trivial and hence (\ref{eqn:Inj}) injective, see 
 Proposition \ref{prop:TSEll}. But we expect that both obstructions vanish also
 for higher-dimensional linear systems, for which one would need to extend
 them to $|h|$ and then use that $\Br(|h|)$ is trivial.
 
\smallskip

\begin{prop}\label{prop:inj}
Every element in the kernel of  (\ref{eqn:Inj}) divides $e^2$, where
$e$ is the positive generator of $(\NS(S).h)$.
\end{prop}

\begin{proof} Assume $\alpha\in\SBr(S,h)$ is an element in the kernel of (\ref{eqn:Inj}). The first part of the following argument works for any $\alpha\in \SBr(S,h)$

The relative moduli space $\Pic_\alpha^0\coloneqq \Pic_\alpha^0(\kc/|h|_{\rm sm})$ over $P=|h|_{\rm sm}$ is not necessarily fine. So, a universal bundle
$\kp\to \Pic^0_\alpha\times_P\kc$ only exists as a $(\beta,\bar\alpha)$-twisted
sheaf for some $\beta\in \Br(\Pic_\alpha^0)$. Note that the order of $\beta$
divides $e$. Indeed, any divisor $D$ on $S$ with $(D.h)=e$ defines
a relative cycle $\tilde P$ of degree $e$ of $\kc\to P$ on which every Brauer
class on $\kc$ coming from $S$ vanishes. This applies to $\alpha$
and implies that the base change $\Pic_\alpha^0\times_P\tilde P$ is
a fine moduli space and hence $|\beta|$ divides $e$. 

On the other hand, we assume that $\Pic_\alpha^0\to P$ is a trivial torsor
and, therefore, comes with a section $P\subset \Pic_\alpha^0$. The restriction
of $\kp$ to $P\times_P\kc$ is then locally free of rank one twisted with respect to
$(\beta|_P,\alpha)=\pi^\ast(\beta|_P)\cdot\bar\alpha$. In other words,
$\bar\alpha=\pi^\ast(\beta|_P)^{-1}$ and, therefore, the order of $\bar\alpha\in \Br(S)$ divides $e$. Hence, the order of $\alpha\in \SBr(S)$ divides $e^2$.
\end{proof}

Note that in the final part of the proof, proving injectivity under the assumption on the order
of $\alpha$, we only had to address (i). The hypothesis of Proposition
\ref{prop:inj} can be slightly
weakened by also addressing (ii) while using that the untwisted
$\Pic^d$ (as an open subset of the moduli space
$M(0,h,d+(1-g))$ on $S$) is a fine moduli space  if $d+(1-g)$ and $2g-2$ 
coprime
(or, slightly weaker, if $d+(1-g)$ and $e$ are coprime). We leave this to the reader.}\medskip

\noindent{\it Proof of Theorem \ref{thm:1}.}
We conclude the proof of Theorem \ref{thm:1} by combining Proposition \ref{prop:comptwo} and Remarks \ref{rem:ShahBr} \& \ref{rem:groupstructure}.\qed

\subsection{Analytic Tate--{\v{S}}afarevi{\v{c}} group}\label{sec:Markman}
Before continuing our discussion we make a digression on the analytic version of our construction. \New{We introduce the analytic Brauer  of a complex projective K3 surface $S$ }and compare it to a construction of Markman \cite{Mark},
{cf.\ \cite{AbRo}}. This will not be used in the rest of the paper. \New{We will restrict to the case that $h$ is primitive, which allows us to
view $\Sha(S,h)$ as a quotient of $\SBro(S)$, cf. \S\! \ref{sec:mainproof} and Remark
\ref{rem:reiter}}.\smallskip

In our context, it seems natural to introduce the \emph{analytic special Brauer group} as
$$\SBro(S)^{\rm an}\coloneqq H^2(S,\ko_S)/T(S).$$ 
It comes with a surjection onto the analytic Brauer group $$\SBro(S)^{\rm an}= H^2(S,\ko_S)/T(S)\twoheadrightarrow \Br(S)^{\rm an}\coloneqq H^2(S,\ko_S^\ast)$$ and 
naturally contains the (algebraic) special Brauer group as the subgroup of all
torsion elements $\SBro(S)\subset\SBro(S)^{\rm an}$. We have the natural commutative diagram
$$\xymatrix{0\ar[r]&A(S)\ar[r]\ar@{=}[d]&\SBro(S)\cong T(S)\otimes\QQ/\ZZ\ar[r] 
\ar@{^(->}[d]&\Br(S)\cong T(S)^\ast\otimes\QQ/\ZZ\ar[r]\ar@{^(->}[d]&0\\
0\ar[r]&A(S)\ar[r]&\SBro(S)^{\rm an}\cong H^2(S,\ko_S)/T(S)\ar[r]&\Br(S)^{\rm an}=H^2(S,\ko_S^\ast)\ar[r]&0.}$$

Also, analogously to (\ref{eqn:SBroSha}), one can define the analytic version $\Sha(S,h)^{\rm an}$ of $\Sha(S,h)$
as a quotient $\SBro(S)^{\rm an}\twoheadrightarrow  \Sha(S,h)^{\rm an}$. Then
$\Sha(S,h)\subset\Sha(S,h)^{\rm an}$ is the torsion subgroup. The corresponding commutative diagram is
$$\xymatrix{0\ar[r]&\ZZ/(\NS(S).h)\ar[r]\ar@{=}[d]&\Sha(S,h)\ar[r] 
\ar@{^(->}[d]&\Br(S)\ar[r]\ar@{^(->}[d]&0\\
0\ar[r]&\ZZ/(\NS(S).h)\ar[r]&\Sha(S,h)^{\rm an}\ar[r]&\Br(S)^{\rm an}\ar[r]&0.}$$

It is possible to introduce $\Pic_\alpha^0(\kc/|h|_{\rm sm})$ for any class in the
analytic special Brauer 
$\alpha\in \SBro(S)^{\rm an}$ or $\alpha\in \Sha(S,h)^{\rm an}$. However, if $\alpha$ is not torsion, then it will be non-algebraic and, in particular, it is more difficult to talk about its
generic fibre. In other words, only a bimeromorphic equivalence class 
of a complex manifold together with a holomorphic projection to $|h|$ will be defined.
\smallskip

 Markman introduces a group called $\Sha^0$, see \cite[(7.7)]{Mark}. He
identifies it with a certain finite quotient of $\SBro(S)^{\rm an}=H^2(S,\ko_S)/T(S)$
by enlarging the transcendental lattice $T(S)$ by all classes in $T(S)\otimes\QQ$ that can be completed  by elements in $\NS(S)\otimes\QQ$ to 
integral classes in $H^2(S,\ZZ)$ that are  orthogonal to all irreducible components
of all curves $C\in |h|$. In particular, if $S$ has Picard number one, then $\SBro(S)^{\rm an}=\Sha^0$, but in general $\Sha^0$ depends on the choice of $h$ and approximates
our analytic Tate--{\v{S}}afarevi{\v{c}} group $\Sha(S,h)^{\rm an}$. More precisely, the analytic version of the quotient (\ref{eqn:SBroSha}) factors through $\Sha^0$:
$$\SBro(S)^{\rm an}\twoheadrightarrow \Sha^0\twoheadrightarrow \Sha(S,h)^{\rm an}.$$
\New{For non-primitive complete linear system $h=k\cdot h_0$ the discussion
shows that the torsion group of Markman's $\Sha^0$ maps onto
the subgroup $\Sha(S,h_0)\subset\Sha(S,h)$ of index $k$.}


\subsection{Digression on B-fields}\label{sec:Bfield2}
We briefly come back to Remark \ref{rem:Bfield1}.  
Assume $\alpha\in \Br(S)\cong H^2(S,\ko_S^\ast)_{\rm tor}\cong T'(S)\otimes \QQ/\ZZ$ and
pick a B-field lift of $\alpha$, i.e.\ a class $B\in H^2(S,\QQ)$ such  that
its image under the exponential map $H^2(S,\QQ)\to H^2(S,\ko_S)\to H^2(S,\ko_S^\ast)$,  or, equivalently, under the projection $H^2(S,\QQ)\twoheadrightarrow T'(S)\otimes\QQ\twoheadrightarrow T'(S)\otimes\QQ/\ZZ$, gives back $\alpha$.

Using the decomposition $H^2(S,\QQ)=(\NS(S)\otimes\QQ)\oplus (T(S)\otimes\QQ)$, the class $B$ can also be projected onto a class in $T(S)\otimes \QQ$ and
then further onto a class $\tilde\alpha\in\SBro(S)\cong T(S)\otimes\QQ/\ZZ$,
which maps to $\alpha$ under $\SBro(S)\twoheadrightarrow \Br(S)$. In general, picking a B-field lift for
a class $\alpha\in \Br(S)$ is strictly more information than is actually needed. We will now explain why for most practical purposes, a lift to a class in $\SBro(S)$ suffices.\smallskip

(i) Via the exponential map, the class $\alpha\in\SBr(S)$ or rather the corresponding class in $H^2(S,\ZZ)\otimes\QQ/\ZZ$ maps to a class $\alpha^{0,2}\in H^{0,2}(S)=H^2(S,\ko_S)$. This allows one to associate with a generator $\sigma\in H^{2,0}(S)=H^0(S,\omega_S)$ the class $\sigma_\alpha\coloneqq\sigma+\sigma\wedge \alpha^{0,2}\in H^2(S,\CC)\oplus H^4(S,\CC)$, so that we can define a Hodge structure,$$\widetilde H(S,\alpha,\ZZ)$$ of K3 type associated
with $\alpha\in\SBro(S)$  as follows:
The underlying lattice is nothing but the extended Mukai lattice $\widetilde H(S,\ZZ)$ and its $(2,0)$-part is spanned by $\sigma_\alpha$. All other parts of the Hodge structure are then determined by the usual orthogonality requirements. See \cite{HS1,HSeattle}.

\smallskip

(ii)  Representing the  class $\alpha\in\SBr(S)$ in the special Brauer group by
an Azumaya algebra $\ka$, allows us to realise the abelian category $
\Coh(S,\alpha)$ as $\Coh(S,\ka)$. Then, for any $E\in \Coh(S,\alpha)$ we
consider the Mukai vector $v_\alpha(E)=\ch_\alpha(E)\cdot\sqrt{{\rm td}(S)}$,
well defined and independent of the choice of $\ka$, see Definition--Proposition \ref{defprop:Chern}. But the Mukai vector is of type $(1,1)$ with respect to the untwisted Hodge structure and in general not with respect to the twisted Hodge structure. However, things
are better in our situation, as indeed $$v_\alpha(E)\in\widetilde H^{1,1}(S,\alpha,\QQ)$$ for sheaves $E\in \Coh(S,\alpha)$ supported on curves in $S$.

\section{Elliptic and Lagrangian fibrations}
We specialise to the case of elliptic K3 surfaces and link our theory to the classical
Ogg--Tate--{\v{S}}afarevi{\v{c}} theory, see  \cite[Ch.\ 11]{HuyK3} for comments and references. {We wish to point out that moduli spaces of twisted
line bundles on fibred surfaces have been studied by Lieblich in his thesis
\cite[Ch.\ 5]{LiebPhD}.} \smallskip

For an elliptic K3 surface $S_0\to \PP^1$ with a section, the Tate--{\v{S}}afarevi{\v{c}} group $\Sha(S_0/\PP^1)$ parametrises
 elliptic K3 surfaces $S\to\PP^1$ together with an isomorphism $S_0\cong\overline\Pic^0(S/\PP^1)$ over $\PP^1$. The surface $S_0$ is thus a moduli space of stable sheaves on $S$. However, usually it is only a coarse moduli space, i.e.\ a universal sheaf on $S_0\times_{\PP^1}S$ exists only as a twisted sheaf with respect to a Brauer class $\beta\in \Br(S_0)$. Turning this around, shows that the twists $S$ parametrised
 by $\Sha(S_0/\PP^1)$ can be viewed as moduli spaces of twisted sheaves  
of rank one on the fibres of $S_0\to\PP^1$. One would like to write this last observation as $\Pic^0_\beta(S_0)\cong S$,
 although $\beta$ seems to occur naturally as a class in $\Br(S_0)$ and not
 as a class in  $\SBro(S_0)$. The reason why this is possible is that the natural
 map $\Sha(S_0,f)\twoheadrightarrow \Br(S_0)$ is in fact an isomorphism
 and $ \Sha(S_0,f)\cong \Sha(S_0/\PP^1)$. This shall be explained first.

\subsection{Comparison with the classical Tate--{\v{S}}afarevi{\v{c}} group}\label{sec:classTS}
We begin with an elliptic K3 surface $S\to \PP^1$ without a section. We
denote by $f$ the class of a fibre and let $m$ be such that
$m\ZZ=(\NS(S).f)$. Then $S_0\coloneqq\overline\Pic^0(S/\PP^1)=M(0,f,0)$ is the relative Jacobian of $S$, which is an elliptic K3 surface
$S_0\to\PP^1$ with a section.

Denoting by $f$ also the class of a fibre of $S_0\to\PP^1$,
the existence of a section implies $(\NS(S_0).f)=\ZZ$.
Then, by virtue of Remark \ref{exa:divone}, the natural projection
in (\ref{eqn:CDShah}) is an isomorphism
\begin{equation}\label{eqn:ShaBr1}
\Sha(S_0,f)\congpf\Br(S_0).
\end{equation}
On the other hand, viewing $\Pic_\beta^0(S_0/\PP^1)$ as a torsor for {$\Pic^0(S_0/\PP^1)$, which compactifies to $S_0$,} defines a group homomorphism
$\Sha(S_0,f)\to\Sha(S_0/\PP^1)$, which is in fact an isomorphism.

\begin{lem}\label{lem:ShaSha}
Mapping $\beta\in \Sha(S_0,f)$ to  $\Pic^0_\beta(S_0/\PP^1)$ viewed as
a torsor for $S_0$ defines an isomorphism of groups
\begin{equation}\label{eqn:ShaSha}
\Br(S_0)\cong\Sha(S_0,f)\congpf\Sha(S_0/\PP^1).
\end{equation}
\end{lem}

\begin{proof}
 Indeed, as recalled above, we know that every twist $S$ is
a moduli space of twisted sheaves of rank one on the fibres of $S_0\to\PP^1$.
Due to the existence of the section,  $\Pic_\beta^d(S_0/\PP^1)\cong\Pic^0_\beta(S_0/\PP^1)$ for all $d$, so the map
is surjective.

The injectivity is a consequence of \New{the arguments before Proposition \ref{prop:inj}.}
\end{proof}

Combining (\ref{eqn:ShaBr1}) and (\ref{eqn:ShaSha}), one recovers directly the well-known isomorphism
\begin{equation}\label{eqn:ShaBrnew}
\Sha(S_0/\PP^1)\cong\Br(S_0).
\end{equation}
However, at this point it is not evident that our isomorphism coincides with
classical one by Artin and Tate \cite[\S4]{BrauerII} that uses the \New{Leray} spectral sequence, cf.\  Remark \ref{rem:ShaBrnew}.

The classical Tate--{\v{S}}afarevi{\v{c}} group $\Sha(S_0/\PP^1)$ is only defined
for elliptic K3 surfaces with a section. But the description of it as
$\Sha(S_0,f)$ allows us to speak of the Tate--{\v{S}}afarevi{\v{c}} group without
assuming the existence of a section.

\begin{definition} Assume $S\to\PP^1$ is an elliptic K3 surface (with or without a section). If $f$ denotes the class of a fibre, then we call $\Sha(S,f)$ the \emph{Tate--{\v{S}}afarevi{\v{c}} group} of $S\to\PP^1$.
\end{definition}

\subsection{Tate--{\v{S}}afarevi{\v{c}} group without a section} As a next step we will explain the link between this new 
Tate--{\v{S}}afarevi{\v{c}} group $\Sha(S,f)$ of an elliptic K3 surface $S\to \PP^1$
and the  classical Tate--{\v{S}}afarevi{\v{c}} group $\Sha(S_0,f)\cong\Sha(S_0/\PP^1)$ 
of its Jacobian fibration.\smallskip

We begin by recalling the short exact sequence
\begin{equation}\label{eqn:ShaBr}\xymatrix{0\ar[r]&\langle\beta\rangle\ar[r]& \Br(S_0)\ar[r]&\Br(S)\ar[r]&0,}
\end{equation}
where $\beta\in \Br(S_0)\cong\Sha(S_0/\PP^1)\cong\Sha(S_0,f)$ is the class corresponding to $S\to \PP^1$.
The result goes back to Artin and Tate, cf.\ \cite[\S4]{BrauerII}. First consider the scheme-theoretic generic fibre $E$ of $S_0\to\PP^1$ which is a smooth curve
of genus one over $\CC(t)$. Its Jacobian $E_0\coloneqq\Pic^0(E)$ is the identity
component of $\Pic(E)$. More precisely, there exists a short exact sequence
$\xymatrix{0\ar[r]&E_0=\Pic^0(E)\ar[r]&\Pic(E)\ar[r]&\ZZ\ar[r]&0.}$ The relative version
of this is a short exact sequence
$$\xymatrix{0\ar[r]&\ks_0\ar[r]&R^1\pi_\ast\GG_m/_{\text{vert}}\ar[r]&\ZZ\ar[r]&0,}$$
cf.\ \cite[Rem.\ 11.5.9]{HuyK3}. Here, $\pi\colon S\to \PP^1$ is the projection and $\ks_0$ is the sheaf of \'etale (or analytic) local section of the Jacobian fibration $S_0=\Pic^0(S/\PP^1)\to\PP^1$. Taking cohomology leads to the exact sequence
\begin{equation}\label{eqn:ArtTate}
\xymatrix{H^0(\PP^1,R^1\pi_\ast\GG_m)\ar[r]&\ZZ\ar[r]&H^1(\PP^1,\ks_0)\ar[r]&
H^1(\PP^1,R^1\pi_\ast\GG_m)\ar[r]&0.}
\end{equation}
The  \New{Leray} spectral sequence induces isomorphisms
$\Pic(S)\cong H^0(\PP^1,R^1\pi_\ast\GG_m)$ and $\Br(S)\cong H^2(S,\GG_m)\cong H^1(\PP^1,R^1\pi_\ast\GG_m)$. Also, the first map is nothing
but $L\mapsto (L.f)$ and, therefore, its cokernel is
$\ZZ/(\NS(S).f)$. Applied to the Jacobian fibration $S_0\to\PP^1$ itself, (\ref{eqn:ArtTate}) 
induces an isomorphism $H^1(\PP^1,\ks_0)\cong \Br(S_0)$. In general, one obtains a short exact sequence
$$\xymatrix{0\ar[r]&\ZZ/(\NS(S).f)\ar[r]&\Br(S_0)\ar[r]&\Br(S)\ar[r]&0}$$
and the kernel can indeed be shown to be just the subgroup $\langle\beta\rangle$.
For example, the order $|\beta|$ of the subgroup $\langle\beta\rangle$ equals the minimal fibre degree of any line bundle on $S$, i.e.\ $|\beta| \ZZ=(\NS(S).f)$.

\begin{remark}\label{rem:BrTorH1}
(i) Over the generic point of $\PP^1$, the surjection $\Br(S_0)\twoheadrightarrow \Br(S)$ viewed as the map $H^1(\PP^1,\ks_0)\twoheadrightarrow H^1(\PP^1,R^1\pi_\ast\GG_m)$ corresponds to $$H^1(\CC(t),\bar E_0)\to H^1(\CC(t),\Pic(\bar E))$$ which maps a torsor $E'$ for $E_0$ to the torsor
$(\Pic(E)\times E')/E_0$ for $\Pic(E)$.\smallskip

(ii) The map $$\Br(S)\congpf H^1(\PP^1,R^1\pi_\ast\GG_m)\,\hookrightarrow H^1(\CC(t),\Pic(\bar E))$$ is geometrically realised by mapping a Brauer
class $\alpha\in \Br(S)$ to the $\Pic(E)$-torsor $\Pic_{\alpha|_E}(E)$, see Remarks \ref{rem:Brauertorsor} \& \ref{rem:Brauertorsor2},
and Proposition \ref{prop:BrTorsCurves}. {Recall that the Leray spectral sequence for the projection $S\to \PP^1$ restricts to the Hochschild--Serre spectral sequence on the generic fibre.}
\end{remark}

\begin{remark}
An alternative way of constructing (\ref{eqn:ShaBr}) uses Hodge theory.
Indeed, the universal $(\beta,1)$-twisted sheaf on $S_0\times_{\PP^1}S$
induces a Hodge isometry $\widetilde H(S,\ZZ)\congpf \widetilde H(S_0,\beta,\ZZ)$ which restricts to the isometry $T(S)\congpf T(S_0,\beta)=\ker(\beta\colon
T(S_0)\to\QQ/\ZZ)$, cf.\ Section \ref{sec:Bfield2}, \cite[Ch.\ 14.4.1]{HuyK3}
or \cite[\S4]{HS1}. Applying $\Hom(~,\QQ/\ZZ)$ and using $\Br(S)\cong \Hom(T(S),\QQ/\ZZ)$ leads to an exact sequence  $\xymatrix{0\ar[r]&\langle\beta\rangle\ar[r]&\Br(S_0)\ar[r]&\Br(S)\ar[r]&0.}$
This is indeed nothing but (\ref{eqn:ShaBr}), but as we will not use this fact here, 
we do not give a proof.
\end{remark}

 Comparing
(\ref{eqn:ShaBr}) with the bottom sequence  in 
(\ref{eqn:CDShah})  suggests the next result.

\begin{prop}\label{prop:TSEll}
 Mapping $\alpha\in\Sha(S,f)$ to $\Pic_\alpha^0(S/\PP^1)$,
viewed as a torsor for {$\Pic^0(S/\PP^1)\subset S_0$}, defines
an isomorphism $\Sha(S,f)\congpf\Sha(S_0/\PP^1)$ which can be completed
to a commutative diagram
$$\xymatrix{0\ar[r]&\ZZ/m\ZZ\ar[r]\ar[d]_\cong&\Sha(S,f)\ar[d]_\cong\ar[r]&\Br(S)\ar@{=}[d]\ar[r]&0\\
0\ar[r]&\langle\beta\rangle\ar[r]&\Sha(S_0/\PP^1)\ar[r]&\Br(S)\ar[r]&0.}
$$
In particular, the generator of the subgroup $\ZZ/m\ZZ\subset\Sha(S,f)$ is mapped to the class $\beta\in \Sha(S_0/\PP^1)$ corresponding to $S$.
\end{prop}

\begin{proof} \New{Clearly, $\Pic_\alpha^0(S/\PP^1)$ is torsor
for $\Pic^0(S/\PP^1)\cong \Pic^0(S_0/\PP^1)$ and, therefore, defines an element in the Weil--Ch\^atelet group $\text{WC}(S_0/\PP^1)$. But it is in fact contained in the smaller Tate--{\v{S}}afarevi{\v{c}} group 
$\Sha(S_0/\PP^1)\subset\text{WC}(S_0/\PP^1)$. Indeed,  for a generic polarisation the compactification $\Pic^0_\alpha(S/\PP^1)$ by the moduli
space of all $\alpha$-stable sheaves with the given Mukai vector is  a smooth K3 surface. Hence, the map $\Sha(S,f)\to \Sha(S_0/\PP^1)$ is well defined.}\smallskip

The injectivity of the map is a consequence \New{of the comments at the end of the proof of Proposition \ref{prop:inj}} and the surjectivity follows from both groups being divisible of the same rank.\smallskip

Next we show that $\bar 1\in \ZZ/m \ZZ$ is mapped to $\beta$. This is a consequence of the more general observation
made in Example \ref{ex:genexa}. In this particular case it says
that for the class $\alpha\in \SBr(S)$ represented by $\kend(F)$, where
$F$ is locally free of rank $m$ with its determinant satisfying
$(\det(F).f)=m$, there exists an isomorphism
{$\Pic^0_\alpha(S/\PP^1)\cong\Pic^1(S/\PP^1)\subset S$ .}\smallskip

To prove the commutativity of the diagram on the right, it suffices to control the generic fibre. According to Remark \ref{rem:BrTorH1}, the map $$\Sha(S,f)\to \Sha(S_0/\PP^1)\cong\Br(S_0)\to \Br(S)\,\hookrightarrow H^1(\CC(t),\Pic(\bar E))$$  is given by
$\alpha\mapsto \Pic^0_{\alpha|_E}(E)\mapsto (\Pic(E)\times \Pic^0_{\alpha|E}(E))/E_0$ and  the map
$$\SBro(S)\twoheadrightarrow \Br(S)\,\hookrightarrow H^1(\CC(t),\Pic(\bar E))$$ sends $\alpha\in \SBro(S)$ first to $\bar\alpha\in \Br(S)$ and then to $\Pic_{\bar\alpha|_E}(E)$, which by Proposition \ref{prop:BrSBrPictorsor}
is the torsor $(\Pic(E)\times \Pic^0_{\alpha|_E}(E))/E_0$.
\end{proof}

\begin{remark}
Observe that the surjectivity says that for every $\gamma\in \Sha(S_0/\PP^1)\cong\Br(S_0)$ there exists a class $\alpha\in \SBro(S)$ such that
$$\Pic^0_\gamma(S_0/\PP^1)\cong \Pic^0_\alpha(S/\PP^1).$$ It would be interesting to
find a geometric proof for this. One idea could be to use the two twisted
universal families on $S_0\times_{\PP^1} \Pic^0_\gamma(S_0/\PP^1)$ and on
$S\times_{\PP^1} S_0$ to produce a family on $S\times_{\PP^1}\Pic^0_\gamma(S_0/\PP^1)$ inducing the desired isomorphism
by universality. However, as the first family is twisted by $(\gamma\times 1)$ and the
second one by $(1\times\beta)$, they do not concatenate directly. One would first need to
transform e.g.\ the first one to a family twisted by some $(\beta^{-1}\times\delta)$.
\end{remark}

\begin{remark}\label{rem:ShaBrnew}
(i) Observe that the above arguments in particular show that the
isomorphism (\ref{eqn:ShaBrnew}) coincides with the classical one constructed
via the \New{Leray} spectral sequence.\smallskip

(ii) There is yet another way of linking $\Br(S_0)$ and $\Sha(S_0/\PP^1)$, cf.\
\New{\cite[\S 5.4]{CaldPhD} or}
\cite[Rem.\ 11.5.9]{HuyK3}. If $S$ is the twist associated with $\beta\in \Sha(S_0/\PP^1)$, then $S_0$ can be viewed as a moduli space of sheaves on $S$, namely $\Pic^0(S/\PP^1)\cong S_0$. However, $S_0$ is only a coarse moduli space and the obstruction  
 to the existence of a universal family is a class $\gamma_\beta\in \Br(S_0)$. This defines
 a map $\Sha(S_0/\PP^1)\to \Br(S_0)$, $\beta\mapsto \gamma_\beta$, which again is nothing but (\ref{eqn:ShaBrnew}).
\end{remark}

\subsection{Twisting Lagrangian fibrations}\label{sec:Mark}
We conclude by linking our discussion with a result by Markman \cite{Mark}. 

Among other things, he proves that 
every non-special hyperk\"ahler manifold of ${\rm K3}^{[n]}$-type $X$ together
with a Lagrangian fibration $X\to \PP^n$ is a certain `twist' $M_s$ of a Mukai
system $M=\overline\Pic^d(\kc/|h|)\to|h|$ associated with some K3 surface $S$ and a complete
linear system $\kc\to|h|$ on it. His twists are parametrised by elements
$s\in \SBro(S)^{\rm an}$ of the analytic special Brauer group $\SBro(S)^{\rm an}$ (or rather of $\Sha^0$ used in \cite{Mark}, see Section \ref{sec:Markman}), which contains
the special Brauer group $\SBro(S)\subset\SBro(S)^{\rm an}$ as its torsion group.

How do these twists compare to our twisted Picard varieties $\Pic_\alpha^0(\kc/|h|_{\rm sm})$, where $\alpha$ is an element in the smaller group
$\SBro(S)\subset\SBro(S)^{\rm an}$? A priori, our setting seems at the same time more special and more general for the following reasons.\smallskip

(i) Markman only considers relative Picard varieties $\Pic^d(\kc/|h|_{\rm sm})$ which
can be compactfied to hyperk\"ahler manifolds. In other words, only those
moduli spaces $\overline\Pic^d(\kc/|h|)\cong M(0,h,s)$, where $d=s+(1/2)(h.h)$, are considered for which the Mukai vector $(0,h,s)$ is primitive. As we are only concerned with \New{the smooth curves $\kc\to|h|_{\rm sm}$,} this restriction is irrelevant for us.\smallskip

(ii) On the other hand, Markman `twists' an arbitrary smooth $\overline\Pic^d(\kc/|h|)=M(0,h,s)\to|h|$, and not only those with a section as our $\Pic^0(\kc/|h|_{\rm sm})$,
to obtain the given Lagrangian fibration $X\to\PP^n$. This can be remedied
by applying Proposition \ref{prop:getall} and Example \ref{ex:genexa}.

\medskip
Theorem \ref{thm:main2}  reinterprets  \cite[Thm.\ 1.5]{Mark} in the projective setting. \medskip

\noindent{\it Proof of Theorem \ref{thm:main2}.} 
The first step consists of observing
that for $M=\overline\Pic^d(\kc/|h|)$ and $s\in \SBro(S)^{\rm an}$, Markman's twist
$M_s$ is algebraic if  and only if $s$ is contained in $\SBro(S)\subset\SBro(S)^{\rm an}$. For this observe  that $\alpha=s\in \SBro(S)^{\rm an}$ 
is torsion, or equivalently contained in $\SBro(S)$, if and only if its
image $\bar \alpha\in \Br(S)^{\rm an}$ is torsion, i.e.\ contained in $\Br(S)\subset\Br(S)^{\rm an}$. The result was proved in broader generality
by Abasheva and Rogov \cite[Thm.\ 5.19]{AbRo} and \cite[Thm.\ A]{Abash}.
\smallskip

Next, according to Example \ref{ex:genexa}, we know that $M$
is birational to the torsor $\Pic^d(\kc/|h|_{\rm sm})\cong\Pic^0_{\New{-}\bar d}(\kc/|h|_{\rm sm})$ and by virtue of Remark \ref{rem:groupstructure} we have  
\New{\begin{equation}\label{eqn:prod}
\left(\Pic_{-\bar d}^0(\kc/|h|_{\rm sm})\times \Pic^0_{\alpha}(\kc/|h|_{\rm sm})\right)/\Pic^0(\kc/|h|_{\rm sm})\cong\Pic^0_{-\bar d\cdot\alpha}(\kc/|h|_{\rm sm}).
\end{equation}}

To conclude one shows that for $\alpha=s\in \Sha(S,h)$ the twist $M_s$ is 
the left hand side of (\ref{eqn:prod}), which is a direct consequence of the discussion in \cite[\S 7.2]{Mark}. This is proved analogously to the two-dimensional case, cf.\ Remark \ref{rem:ShaBrnew}. Hence, $M_s$
is birational to $\Pic_{-\bar d\cdot\alpha}^0(\kc/|h|_{\rm sm})$.

As an alternative for the last step, one could first prove a version of 
(\ref{eqn:prod}) for the total Picard varieties: 
$\left(\Pic_{-\bar d}^0(\kc/|h|_{\rm sm})\times \Pic_{\alpha}(\kc/|h|_{\rm sm})\right)/\Pic^0(\kc/|h|_{\rm sm})\cong\Pic_{-\bar d\cdot\alpha}(\kc/|h|_{\rm sm})$
and compare the left hand side with $M_s$ via Remark \ref{rem:BrTorH1}, see also Proposition \ref{prop:BrSBrPictorsor}.
\qed

\New{In order to apply the results of Markman, we a priori have to assume in
Theorem \ref{thm:main2} that $X$  is `non-special'. This assumption has subsequently be removed by Abasheva \cite[Thm.\ A]{Abash} and Soldatenkov--Verbitsky \cite[Thm.\ 4.7]{SV}. 

Our assumption  $\rho(X)=2$ implies that $\rho(S)=1$
which ensures that Markman's results can be applied, cf.\ \cite[Ass.\ 7.1 \& Rem.\ 7.2]{Mark}.  While the condition $\rho(X)=2$ determines a complement of a countable union  of  hypersurfaces in the space of all Lagrangian fibred $X$, Markman's condition \cite[Ass.\ 7.1]{Mark} is  Zariski open. Hence, our Theorem \ref{thm:main2} also holds for a Zariski open, dense subset of the space of all 
Lagrangian fibrations $X\to\PP^n$ of $\text{K3}^{[n]}$-type.
 }

\end{document}